\pdfoutput=1
\documentclass[paper=letter, fontsize=10pt,leqno]{scrartcl}
\usepackage[body={5.35in,7.5in}]{geometry} 

\usepackage[english]{babel}				
\usepackage{amsthm}
\usepackage{hyperref}

 \usepackage{verbatim}
\usepackage[pdftex]{graphicx}				
\usepackage{url}
\usepackage{amssymb}
\usepackage{bbm}
\usepackage{MnSymbol}

       \usepackage{dutchcal}

\usepackage[pdftex]{color}

\usepackage{amscd}
\usepackage[format=plain]{caption}
\usepackage{wrapfig}
 \usepackage{float}
 
 \usepackage{bigints}

\usepackage{sectsty}					
\allsectionsfont{\centering \normalfont\scshape}	

\usepackage{mathtools}
\DeclareCaptionStyle{ruled}{labelfont=bf,labelsep=space}
\usepackage[]{forest}

\newtheorem{theorem}{Theorem}

\newtheorem{definition}[theorem]{Definition}
\newtheorem{proposition}[theorem]{Proposition}
\newtheorem{lemma}[theorem]{Lemma}

\newtheorem{example}[theorem]{Example}

\newtheorem*{proposition*}{Proposition}

\title{
		\vspace{-1in} 	
		\usefont{OT1}{bch}{b}{n}
		\normalfont \normalsize \textsc{} \\ [25pt]
		\huge  Rolling systems and their billiard limits
}

\date{}

\author{\normalfont \large 
\ C. Cox\footnotemark[2] \footnote{\scriptsize Department of Mathematics, University of Delaware, Ewing Hall, Newark, DE 19711 },
\  R. Feres\footnote{\scriptsize Department of Mathematics and Statistics, Washington University, Campus Box 1146, St. Louis, MO 63130},
\ B. Zhao\footnotemark[2]
}

 \date{\today}

\begin{document}

\maketitle

\begin{abstract}
\begin{center}
 Abstract \end{center}
{\small  
Billiard systems,
broadly speaking, may be regarded as models of mechanical systems in which rigid parts interact through elastic impulsive (collision) forces.
When it is desired or necessary to  account for linear/angular momentum exchange in collisions involving a spherical body, a type of billiard system often referred to
as {\em no-slip} has  been used. In recent work, it has become apparent that no-slip billiards resemble      non-holonomic mechanical systems in a number of ways. Based on an idea by Borisov, Kilin and Mamaev, we show  that no-slip billiards very generally arise as  limits of non-holonomic (rolling) systems, in a way that is akin to how ordinary billiards arise as limits of geodesic flows through a flattening  of the Riemannian manifold. 
}
\end{abstract}

\section{Introduction}
\label{introduction}
Billiard systems are broadly conceived as a class of mechanical/geometric dynamical systems whose trajectories consist of  continuous,  piecewise free  (or geodesic) motion in the interior of  some configuration  manifold  $M$ (a manifold with boundary and possibly corners), with    instantaneous change of direction due to impulsive forces where the geodesic segments meet the boundary of  $M$. For systems consisting of rigid parts, or masses, these discontinuous changes  result from {\em collisions} between  the parts.
The term {\em geodesic motion} is understood  with respect to  the Riemannian metric on $M$ derived from the system's kinetic energy, hence from a
given distribution of masses.  
Boundary points represent configurations in  which some of the rigid parts are in   contact. 
The  change in velocities, from those immediately prior to a collision event to those immediately after,  is typically  implemented by a linear map defined on the tangent 
space of $M$ at the corresponding boundary point. We call it the {\em collision map} at that point.

For the standard sort of mathematical  billiards, the collision map at a point $q\in \partial M$  is a linear isometry, thus kinetic energy preserving,
 that leaves invariant     the tangential component of  the pre-collision velocity $v\in T_qM$ and flips the sign of the normal component\----that is to say, a specular reflection. The latter condition (of conservation of the tangential component of $v$)   may be interpreted physically as resulting from perfectly slippery contact between the moving parts of the system,  not allowing  for momentum transfer between bodies tangentially to a plane of contact.  
This contact condition does not follow from standard mechanical assumptions involving  conservation laws. In fact, a classification 
of collision maps under the assumptions  of energy conservation, time reversibility, linear and angular momentum conservation (when the appropriate Euclidean symmetries are not broken; they are violated, for example, when one body is held in place), and the very restrictive and natural assumption, typically made implicitly, that the impulsive forces between parts of the system  can only 
act at  the points of contact, was obtained in \cite{CF}. This classification shows  that the space of collision maps can be parametrized, in arbitrary dimensions, by a kind of real Grassmannian variety.  Except in dimension  one, this variety  does not reduce to a single point. For example, in dimension $2$, there are exactly two possibilities: the standard specular reflection and what we have called {\em no-slip} reflection. In higher dimensions, the map is specified by the subspace (in the boundary of $M$) of {\em no-slip contact}, that is, 
the eigenspace of the collision map associated to the eigenvalue $-1$. (Time reversibility implies that the collision map is an involution, hence it can only have eigenvalues $1$ and $-1$.)  In general, we define {\em no-slip} billiards as those  for which the dimension of this subspace of no-slip contact is maximal. 
An explicit description of the no-slip collision map will be given later in this paper. (See Equations (\ref{no-slip collision map}) and (\ref{transformation no-slip}).)

For our main result, showing that no-slip billiards arise as a limit of a rolling system, we   will  take  $M$ to consist of the configurations
of a ball in    the closure of  a domain $P$ (a non-empty, connected open subset) in Euclidean space with piecewise smooth boundary. 
  By a {\em no-slip billiard} in $P$ we will understand  the system in which the ball moves freely inside   $P$
  and undergoes a velocity change at  the boundary according to the collision maps  defined  by Equation (\ref{transformation no-slip}). (The physical/geometric justification of the choice of collision map can be found in \cite{CF}.)  These systems have been used in the past, typically in dimension $2$,     whenever there is the need to account  for the change in angular velocity of
  a spherical body at collisions. (The term {\em rough collisions} instead of {\em no-slip} is also  used.)
  See, for example, \cite{bgutkin,CCCF,CFII,CFZ,Garwin,hefner,LLM,MLL,W}. (In his 1969 paper \cite{Garwin}, {\em Kinematics of an Ultraelastic Rough Ball},   the celebrated physicist Richard Garwin uses the no-slip model to explain the odd bouncing  behavior of a Wham-O Super Ball${}^\copyright$  thrown under a table.) A more systematic study, from a dynamical systems theory viewpoint,  was initiated in \cite{CFZ}, and in \cite{CCCF} it is shown that no-slip billiards share many properties with non-holonomic systems, although a direct connection between no-slip billiards and non-holonomic systems was not made there. More specifically, it is shown in \cite{CCCF} that no-slip billiard systems,
   possibly under the action of a constant external (gravity) force, behave for certain initial conditions like a rolling system in the limit  as the duration of 
the between-collisions steps approaches zero. Thus, in \cite{CCCF}, one obtains in a very special context non-holonomic (rolling) systems from no-slip billiards, whereas here we do the opposite.

  In \cite{BKM}, the authors  derive via a limiting procedure the no-slip billiard reflection map in disc and infinite strip domains  in the plane  from the rolling motion  of a ball on elliptic cylinders and ellipsoids in $\mathbb{R}^3$. Thus they  establish a direct link between no-slip billiards and non-holonomic systems, albeit for very special  shapes of the billiard table.
  The resulting  system is called in \cite{BKM} {\em non-holonomic billiards}.
 It is then 
 natural  to ask  whether 
  no-slip billiards can be obtained from rolling systems in greater generality. 
  Notice that this correspondence between rolling on surfaces and no-slip billiards on  domains in the plane   is similar to the relation between 
  geodesic flows on surfaces and 
   ordinary billiard systems  on domains obtained by a flattening of the surface. (An early  mention of this relation is  \cite{A}, page 184; see also
   \cite{K}. This fact  
       is also a corollary of  Theorem \ref{rolling limit}.)

The present paper shows that no-slip billiards arise from  such  rolling    on submanifolds of  Euclidean space under very general conditions. The limit procedure we adopt  is different from the one used in   \cite{BKM}.  It will become apparent  that these rolling systems (not only for domains in $\mathbb{R}^n$, but for fairly general submanifolds of Euclidean space)   define one-parameter deformations of geodesic flows that depend on both the intrinsic and extrinsic submanifold geometry,   the deformation parameter being   the rolling ball's moment of inertia. 
The no-slip billiard system appears in the limit of the rolling motion as the radius of the  ball approaches $0$.  We call these deformations of geodesic flows {\em rolling flows} on {\em pancake manifolds}. (They should not be confused with the {\em non-holonomic geodesic flows} considered, for example, in  \cite{Arnold}.)

 \begin{figure}[htbp]
\begin{center}
\includegraphics[width=3.0in]{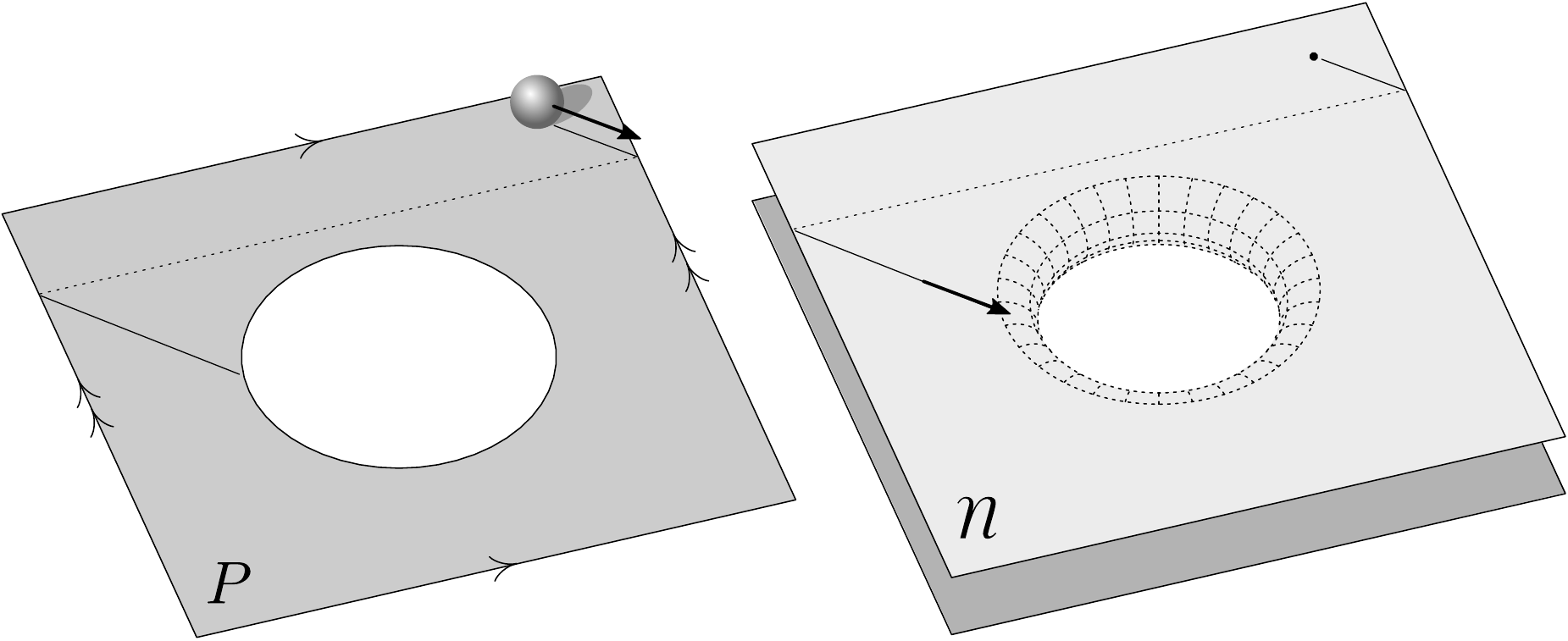}\ \ 
\caption{{\small Rolling system on a Sinai billiard plate and the corresponding pancake surface $\mathcal{N}$, which is the locus of the center of the
rolling ball of radius $r$. }}
\label{sinai pancake}
\end{center}
\end{figure}  

Prior to properly stating the   main result, and in the way of motivating the subject, it may  be helpful to have in mind a couple of concrete examples of rolling systems on a pancake surface.   
 Let us consider first the system defined by a {\em Sinai billiard} plate $P$. (Figure \ref{sinai pancake}.) This is a submanifold of $\mathbb{T}^2\times \mathbb{R}$ (the first factor  being the flat $2$-torus) rather than Euclidean space, but our subsequent discussion is easily adapted to this case.   $P$ is a $2$-torus  with a 
 disc-shaped hole in the middle. 
  The  ball is constrained to  move so as not to lose  contact with $P$ and to roll on it without slipping, which means that
 the point on the ball in contact with the plate at any given moment has zero velocity. No forces, such as gravity, are taken into account except  those enforcing the constraints. The associated pancake surface is the  boundary of the region in $\mathbb{T}^2\times \mathbb{R}$
 consisting of points that lie at a distance no greater than $r$ from $P$. This  surface, denoted $\mathcal{N}$, is differentiably embedded if $r$ is sufficiently small. ($\mathcal{N}$ fails to be smooth on the pair of circles separating the flat and negatively curved regions.)

 A  {\em state} of the system at  a given moment of time  is  the set of positions and velocities that uniquely specifies a trajectory through an initial value problem for Newton's differential equation.
 Each state  consists of the position of the center of the ball (a point in $\mathcal{N}$) and three velocity components: two for
 the velocity of the center of mass, which we call the {\em center velocity} (we always assume that the mass distribution is rotationally symmetric) and one for the angular velocity component
 about the unit normal vector $\nu$ to $\mathcal{N}$ (pointing out), which we call the ball's {\em tangential spin}, or simply {\em spin}. The other  angular velocity components are not independent due to the no-slip rolling constraint.  These three velocity components are further restricted by conservation of kinetic energy.  In higher dimensions, the tangential spin will be a skew-symmetric linear map on $T_x\mathcal{N}$.

 Let us call $u\in T_x\mathcal{N}$ the center velocity vector, $\mathcal{s}$ the spin, $\alpha$ the area $2$-form and $J$ the standard complex structure 
 on $\mathcal{N}$. Thus $J_x$, at each $x\in \mathcal{N}$, is the positive  rotation by $\pi/2$ (relative to the choice of orientation set by $\nu$). As we will see,  the equations of motion when the ambient space is $3$-dimensional Euclidean space
 can be written as
 \begin{equation}\label{roll eq dim 3}\frac{\nabla u}{dt} = -\eta \mathcal{s} J\mathbb{S}_x u, \ \ \dot{\mathcal{s}} = \eta \alpha\!\left(\mathbb{S}_xu,u\right) 
 \end{equation}
 Here $\mathbb{S}_x:T_x\mathcal{N}\rightarrow T_x\mathcal{N}$ is the shape operator of $\mathcal{N}$ at $x$ and $\nabla$ is the surface's Levi-Civita connection.  The parameter $\eta\in [0,1)$ is a function of the ball's moment of inertia. When $\eta=0$, the tangential spin $\mathcal{s}$ is constant and
 the center follows a path with zero acceleration ($\nabla u/dt=0$), which is to say, a geodesic path.

   \begin{figure}[htbp]
\begin{center}
\includegraphics[width=3.0in]{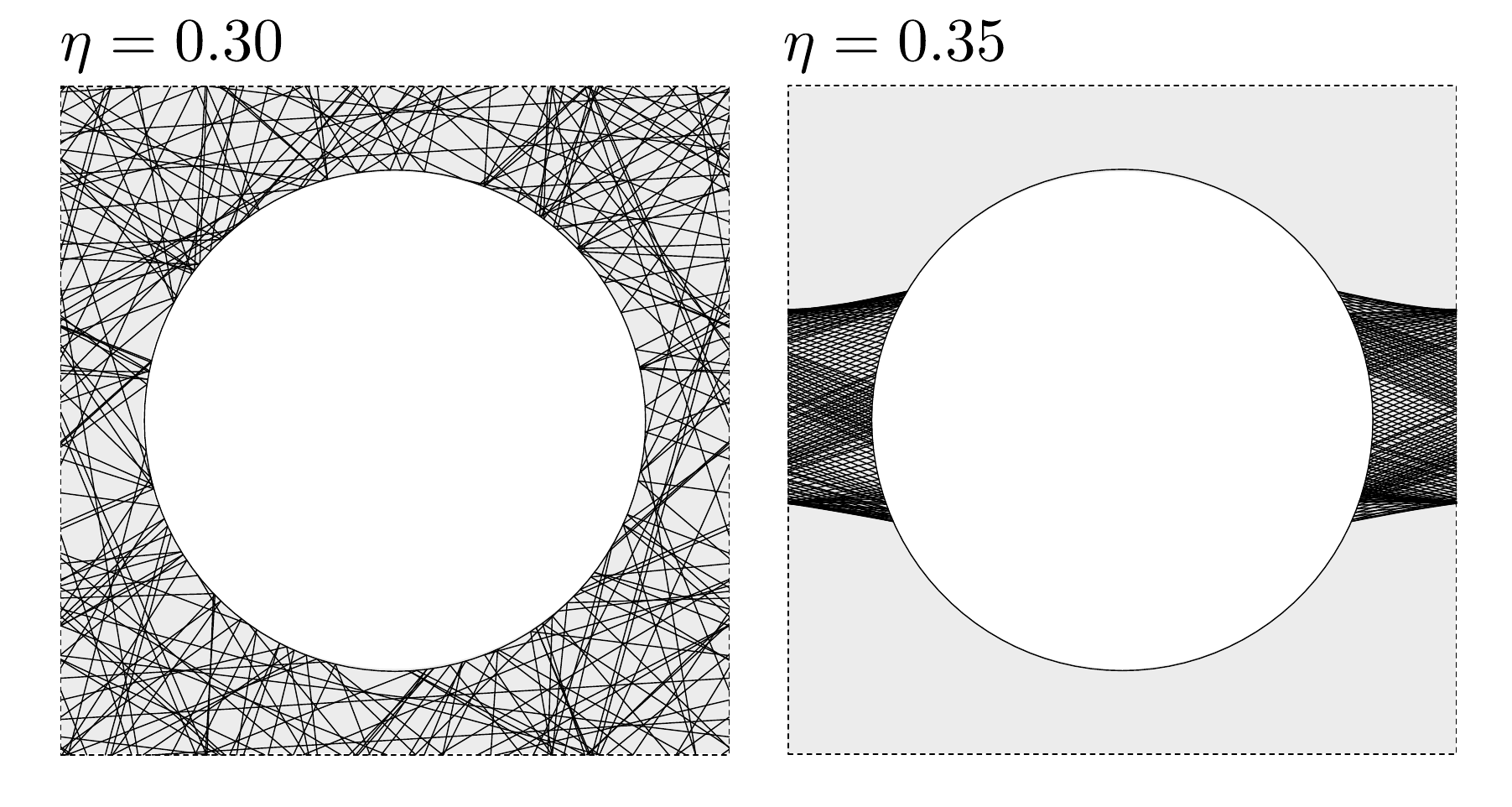}\ \ 
\caption{{\small Two trajectories of the Sinai rolling system with the same initial conditions but different values of the moment of inertia parameter $\eta$. }}
\label{sinai orbits}
\end{center}
\end{figure}

 Figure \ref{sinai orbits} shows  sample trajectories of the center of the ball  for two different values of $\eta$. (The   radius of the ball in this example is too  small for the rounded edge of the circular hole  to  be noticeable in the figure.) It is interesting to note that, for sufficiently small $\eta$, trajectories appear chaotic, which is to be expected since the  geodesic flow on such a non-positively curved surface is known to be chaotic due to the classical work by Y. Sinai. But for sufficiently large $\eta$ the system shows strong stability (elliptic behavior). Results from
 \cite{CFZ} suggest that this system is never ergodic when $\eta>0$, something we cannot prove at this moment.

 Another numerical example is illustrated in Figure \ref{pancake disc}. The flat plate  $P$ is now a disc in $\mathbb{R}^3$ and $\mathcal{N}$ is 
 a circular pancake with width $2r$, for a small value of $r$. When the moment of inertia parameter $\eta$ is $0$, trajectories of the center of the rolling ball
 (viewed from above so that the two flat sheets of $\mathcal{N}$ are seen as one) are nearly the same as the  trajectories for the ordinary billiard on a disc and exhibit  the characteristic caustic circle. For a positive $\eta$, caustics split into two concentric  circles as seen on the right-hand side of Figure \ref{pancake disc}.  This property of the rolling flow in circular plates  is easy to establish in the limit as the radius of the ball is sent to $0$, but it is a more challenging problem, and a topic for future work, to ascertain it for a  positive radius.

  \begin{figure}[htbp]
\begin{center}
\includegraphics[width=3.0in]{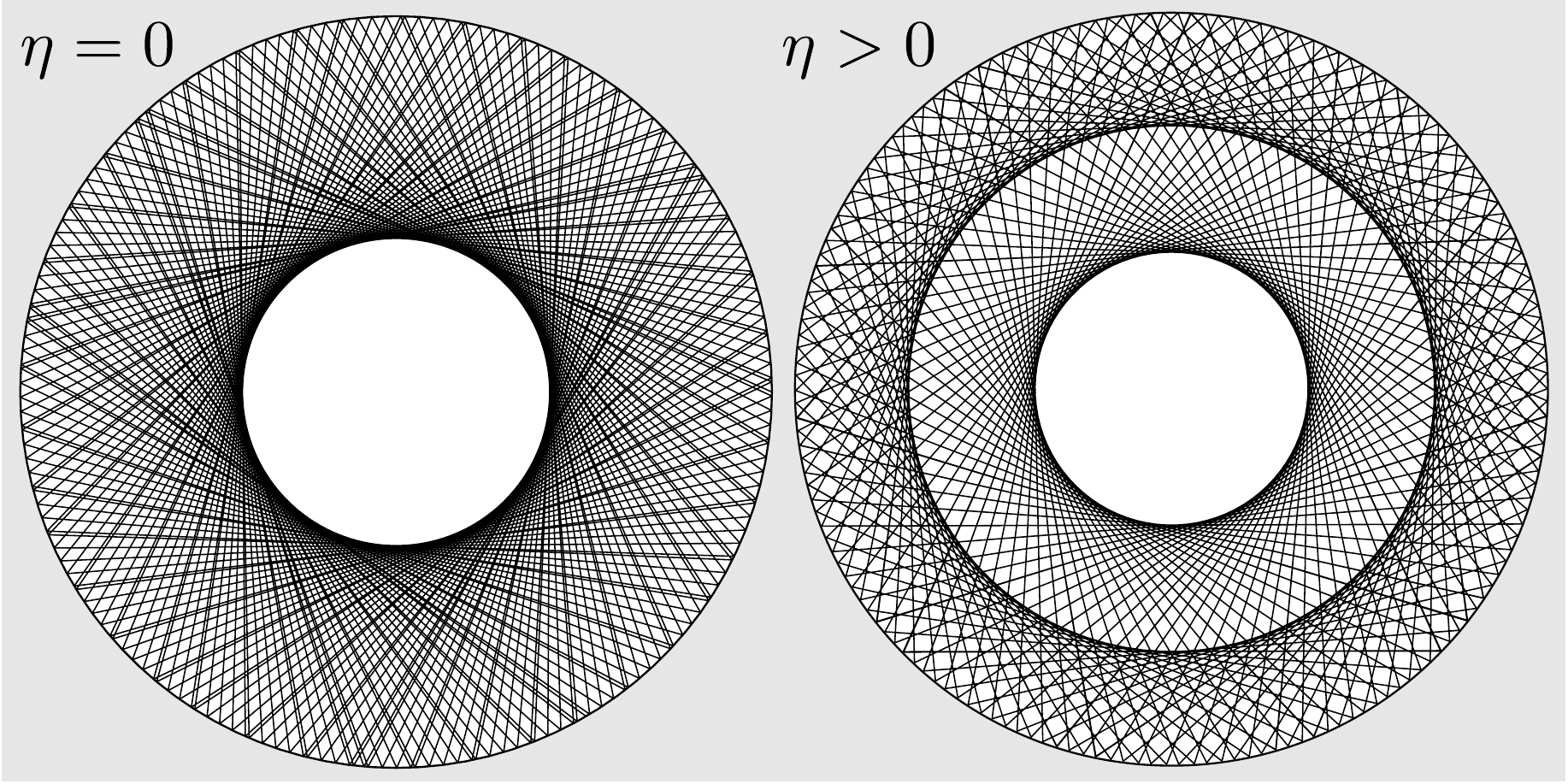}\ \ 
\caption{{\small  Rolling on a disc shaped plate.}}
\label{pancake disc}
\end{center}
\end{figure}  

The rest of this paper is organized as follows. In Section \ref{limit section} we define the no-slip collision map, rolling billiards, and state the result that
the former arises from the latter in the limit as the radius of the rolling ball approaches $0$. (Theorem \ref{rolling limit}.) In Section \ref{flows and volume invariance}  we describe the equations of the rolling flow.
In Section \ref{elementary examples} we give a few examples of rolling systems that can be solved analytically. They all correspond
to  rolling on $P=\mathbb{R}^k$ inside $\mathbb{R}^m$ for different codimensions.
 One of this class of examples (rolling on $\mathbb{R}^k$ inside $\mathbb{R}^{k+2}$) is needed in the proof of the billiard limit result.   The  interest in these
examples lies in   that they provide building blocks for rolling systems on polygonal and polyhedral regions, a topic we plan to return to in a future work.  
 (See \cite{CFZ} for a study of stability properties of no-slip billiards on polygons;  due to Theorem \ref{rolling limit}, results in that paper immediately apply
 to rolling billiards, but the case of positive radius is more challenging.) Section \ref{derivation of newton equation} reviews the derivation of Newton's equation for the non-holonomic system of a ball rolling without slip on   submanifolds of  $\mathbb{R}^m$, and expresses the equations of motion 
 in a form that makes the connection with geodesic flows more explicit. Briefly, the geodesic flow of a Riemannian manifold $\mathcal{N}$ is a flow on the tangent bundle $T\mathcal{N}$;
 the {\em rolling flow} is a flow on    $T\mathcal{N}\oplus \mathfrak{so}(\mathcal{N})$, where $ \mathfrak{so}(\mathcal{N})$ is the vector  bundle   whose elements at a base point $x$ are the skew-symmetric linear maps  from $T_x\mathcal{N}$ to itself. We call the elements of $ \mathfrak{so}(\mathcal{N})$ the 
 {\em tangential spin}.
 When the moment of inertia parameter $\eta$ is zero, the rolling flow reduces to geodesic flow on $T\mathcal{N}$ while elements of $ \mathfrak{so}(\mathcal{N})$ are parallel transported along geodesic flow lines. 
  Finally, in Section \ref{bill limit section} we prove Theorem \ref{rolling limit}. The case $\eta=0$ may be contrasted with the so-called {\em  frame flows}.
  These are flows on the orthogonal frame bundle of a Riemannian manifold whose flow lines project onto geodesics on the manifold and frames
  are parallel transported along these geodesics.  See, for example, \cite{B} and \cite{BG}. What we call tangential spin may  be roughly associated with  the rate of rotation of a moving frame. 
  
  From a dynamical systems and ergodic theory perspective, there are many natural questions one can ask about rolling flows, motivated by 
  related issues  for no-slip billiards and geodesic flows. As already noted, establishing similar results as those from \cite{CFZ} (concerning no-slip billiards on
  polygons) would likely  lead to useful insights, as would understanding the breakdown of ergodicity for sufficiently large moment of inertia parameter $\eta$ for deformations
  of hyperbolic geodesic flows, or the effect of the same parameter on Lyapunov exponents and the metric entropy.  A first step in taking the subject along this direction is to investigate whether the canonical (Liouville) volume form
  defined on the system's phase space   is invariant under the rolling flow.  This is true in dimension $3$ (we do not prove this here) but the general dimension case is work in progress. 
 
  \section{Rolling billiards and no-slip billiards} 
  \label{limit section}
  Before stating the  main result of the paper, which establishes the equivalence between rolling and no-slip billiards, we need to introduce the essential definitions. We begin with a characterization of the no-slip collision map. (For more details, see \cite{CF} and \cite{CCCF}.)

\subsection{No-slip billiards}
Let $P$ be the closure of a domain in $\mathbb{R}^k$.  The center of a  ball of radius $r$ contained in $P$ must clearly be at a distance at least $r$ from
$\partial P$. To avoid adding more notation, we deviate slightly from above  description  in Section \ref{introduction} and assume from now on that $P$ is the locus of centers of the ball rather than the full billiard table. The actual billiard table is then the union of  $P$ and the tubular neighborhood of its boundary $\partial P$. We call this union the {\em extended billiard table}. 
We assume
that $P$  has piecewise smooth boundary. Let $\mathbbm{n}(x)$ be the inward pointing unit normal vector  defined at a point $x$ on 
the smooth part of the boundary. 

The configuration manifold of a ball moving freely in the interior of $P$ (no non-holonomic constraints here!) is the subset $M=\{(x,A)\in SE(k): x\in P\}$
of the Euclidean group $SE(k)$. Here we are denoting by $(x,A)$ elements of the Euclidean group, where $x$ is the  translation part and $A$ the rotation.
The boundary of  $M$ consists of the elements $q=(x,A)$ where $x$ lies on the boundary of $P$. 
A subbundle $\mathfrak{R}$ of the tangent bundle to the boundary of $M$ can be defined by the linear condition that a vector $v\in \mathfrak{R}_{q}$ describes a state in which the point on the ball's surface in contact with the boundary of the extended billiard table has zero velocity.  
We denote tangent vectors to $M$ by $(u,S)\in T_xP\times \mathfrak{so}(k)$, where $u$ is the center velocity and $S=\dot{A}A^{-1}$ is the angular velocity matrix, which lies in the Lie algebra $\mathfrak{so}(k)$ of the rotation group $SO(k)$. 
 The kinetic energy Riemannian metric on $M$ has the following form. Let $\xi = (u_\xi, S_\xi), \zeta=(u_\zeta, S_\zeta)$ be tangent to $M$ at $(x, A)$; then
 $$\langle \xi,\zeta\rangle = \mathcal{m}\left\{\frac{\left(r\gamma\right)^2}{2}\text{Tr}\left(S_\xi S_\zeta^\intercal\right) + u_\xi\cdot u_\zeta\right\}, $$
where
$u_\xi\cdot u_\zeta$ stands for the ordinary dot product, 
 $\mathcal{m}$ is the total mass of the ball, and $\gamma$ is a parameter related to moment of inertia (obtained from the matrix of second moments of
the mass distribution, which must be a scalar matrix  under the assumption that this distribution is rotationally symmetric). The parameter $\eta$ mentioned earlier
is $\eta=\gamma/\sqrt{1+\gamma^2}$. Both $\gamma$ and $\eta$ are independent of the radius $r$. 
 It is also useful to introduce the numbers
$$c_\beta:= \cos\beta :=\frac{1-\gamma^2}{1+\gamma^2}, \ \ 
    s_\beta:= \sin\beta  :=\frac{2\gamma}{1+\gamma^2}.$$

The collision map at a boundary point $q=(x,A)$ of $M$ is a linear map $C_q: T_qM\rightarrow T_qM$ that sends vectors pointing out of $M$ to vectors pointing into it and satisfies certain natural requirements listed in \cite{CF}. 
Those requirements imply that the restriction of $C_q$ to $\mathfrak{R}_q$ must be the identity.
  Among the possible choices reflecting the nature of
contact between the ball and the boundary of $P$ (whether it is ``slippery'' or ``rough''), we call {\em no-slip} that choice which  makes contact
maximally rough.  Mathematically, this means that the restriction of  $C_q$  to the orthogonal complement of $\mathfrak{R}_q$ (with respect to the above Riemannian metric) is minus the identity map.
We refer the reader to \cite{CF} and \cite{CCCF} for further explanations.   The  no-slip collision map can now  be expressed as follows (see Proposition 15 of
\cite{CCCF}):
\begin{equation}\label{no-slip collision map} C_q(u,{S})=\left(  
c_\beta u - \frac{s_\beta}{\gamma} (u\cdot \mathbbm{n}(x))\mathbbm{n}(x) + s_\beta \gamma r S \mathbbm{n}(x),  {S}+ \frac{s_\beta}{\gamma r}\mathbbm{n}(x)\wedge \left[u - r S\mathbbm{n}(x)\right]
\right).\end{equation}
In this expression, $\wedge$ stands for the 
 {\em cross-product} $(a,b)\in \mathbb{R}^k\times  \mathbb{R}^k\mapsto  a\wedge b\in
\mathfrak{so}(k)$, defined by $u\mapsto (a\wedge b)u=(a\cdot u)b-(b\cdot u)a$ and $\mathbbm{n}(x)$, we recall,  is the inward pointing unit normal vector to 
the boundary of $P$ at $x$. 

A minor change in notation will help make the expression on the right-hand side of  Equation (\ref{no-slip collision map}) a little more transparent. We write
$\mathcal{S}=r\gamma S$, so the Riemannian metric simplifies to 
$$\langle \xi,\zeta\rangle = \mathcal{m}\left\{\frac{1}{2}\text{Tr}\left(\mathcal{S}_\xi \mathcal{S}_\zeta^\intercal\right) + u_\xi\cdot u_\zeta\right\}.$$
We also write $W:=\mathcal{S}\mathbbm{n}(x)$
and denote by $\Pi_x$ the orthogonal projection from $\mathbb{R}^k$ to the tangent space of the boundary of $P$  at $x\in \partial P$. Elements of $\mathfrak{so}(k)$ may be decomposed as 
$$\mathcal{S}= \Pi_x\mathcal{S}\Pi_x + \mathbbm{n}(x) \wedge W$$ 
as an elementary calculation can show.  (See Lemma \ref{lemma decomposition}.)   Then the effect of $C_q$ is to map 
 \begin{equation} \label{transformation no-slip}  \Pi_x \mathcal{S}\Pi_x\mapsto \Pi_x \mathcal{S}\Pi_x, \ \   \mathbbm{n}(x)\mapsto -\mathbbm{n}(x),  \ \ \left(\begin{array}{c} \bar{u} \\W\end{array}\right) \mapsto \left(\begin{array}{cr}c_\beta I & s_\beta I \\s_\beta I & -c_\beta I\end{array}\right) \left(\begin{array}{c}\bar{u} \\W\end{array}\right),\end{equation}
where $\bar{u}=\Pi_x u$ and $I$ is the identity map on the tangent space to the boundary of $P$ at $x$. Note that $W=\mathcal{S}\mathbbm{n}(x)$ is in this
tangent space since $\mathcal{S}$ is skew-symmetric.  The orthogonal transformation involving $\bar{u}$ and $W$ given by (\ref{transformation no-slip})  is the characteristic exchange of linear and angular 
velocities of no-slip collisions.  

 We can now define no-slip billiards as the system whose orbits in the interior of $P$ consist of  straight line segments with constant $u$ and constant $\mathcal{S}$ (as justified in  \cite{CF}), and at the boundary undergoes a change of  velocities according to the above collision map $C_q$. When the mass distribution of the ball
 is entirely concentrated at the center, $\gamma=0$ and the collision map reduces to a transformation that decouples linear and angular velocities:
 the center velocity $u$ transforms according to the standard billiard reflection, and the components of the angular velocity contained in $W$ switch sign while the other components remain the same.
 
 Many of the concepts introduce above for no-slip billiards will have their counterparts for rolling billiards. They  will be distinguished  by adding a
 superscript. Thus, for example, we  write $C^b$ and $C^r$, $\gamma^b$ and $\gamma^r$, $W^b$ and $W^r$, and so on. 
 (Here $r$ stands for ``rolling'' and $b$ for ``billiard.'')
 We omit the superscript when the context makes it clear to which type of system  we are referring.

 \subsection{Rolling billiards}
 In order to define rolling billiards  on the $k$-dimensional  $P$   we now regard $P$ as a  flat submanifold of $\mathbb{R}^{k+1}$ (so it is still a domain in $\mathbb{R}^k\subset \mathbb{R}^{k+1}$).   For the rolling to be well defined, we assume (for a given $r>0$) that the boundary $\mathcal{N}$ of the set of points at a distance no greater than $r$ from $P$  has a continuous and piecewise smooth unit normal  vector field. 
 Elements of $\mathbb{R}^k$ will consist of $(k+1)$-tuples with the last component equal to $0$. We wish to recover the no-slip billiard collision map
 in the limit as the radius of a ball rolling around the edge of the plate $P$ goes to zero.

  One point to bear in mind is that the behavior of the rolling ball at the edge of $P$ is not easy to describe explicitly when
 the boundary of $P$ has non-zero curvature. For example, it is perfectly possible for the ball to roll on the edge of $P$ 
 only part of the way, turn around,  and then return to the same side from which it first arrived at the edge. (These two sides correspond to the
 two flat sheets comprising the pancake hypersurface $\mathcal{N}$. See Figure \ref{sinai pancake}.) The ball can also 
 roll on the edge of $P$ for an indefinite amount of time before finally exiting on one of the two sides. In the limit, however, the motion is greatly 
 simplified and the ball necessarily rolls
 to the opposite side of $P$ in a relatively simple motion that can be described analytically.

The notation in the rolling ball set-up is similar to what we used above for the no-slip billiard, but we need to be attentive to the new context. The rolling ball is
 now  $(k+1)$-dimensional while the no-slip billiard ball is  $k$-dimensional. 
 The unit vector $\mathbbm{n}(x)$ is, as before: the inward pointing unit normal vector at a boundary point $x$. Let $U\in \mathfrak{so}(k+1)$ represent the angular velocity matrix of the rolling ball.  Let  $\Pi^\mathcal{N}_x$, $x\in P$, be the orthogonal projection from $\mathbb{R}^{k+1}$ to $T_x\mathcal{N}$. As before, $\Pi_x$ is the orthogonal projection to the tangent space to  $\partial P$ at  a boundary point $x$.
 The parameters $\gamma$ and $\eta$ are defined just as in the no-slip billiard setting, but now they are associated to the mass distribution of  a higher dimensional ball. They will be given superscripts, $\gamma^r, \eta^r$ when we need to make direct comparisons with the corresponding $\gamma^b, \eta^b$.  
 We define $$S=\Pi^\mathcal{N}_x U\Pi^\mathcal{N}_x \ (x\in P), \ \ \mathcal{S}=r\eta S, \ \ W=\mathcal{S}\mathbbm{n}(x)\ (x\in \partial P).$$ Note that  $\mathcal{S}^b=r\gamma^b S^b$ while $\mathcal{S}^r = r\eta^r S^r$. Finally, just as before, $\bar{u}=\Pi_x u$. We call $\mathcal{S}$
 the ball's {\em tangential spin}. We take as a standing assumption that the principal curvatures of the boundary of $P$ are uniformly bounded.

 \begin{theorem} \label{rolling limit}
 In the limit as the radius of the rolling ball  goes to $0$, the velocity components of the  ball immediately before and immediately after rounding the edge of the flat plate $P$  at a boundary point  $x$, are
 related by the   linear map
 $$\Pi_x \mathcal{S}\Pi_x\mapsto \Pi_x \mathcal{S}\Pi_x, \ \   \mathbbm{n}(x)\mapsto -\mathbbm{n}(x),  \ \ \left(\begin{array}{c} \bar{u} \\W\end{array}\right) \mapsto \left(\begin{array}{cr}c & s \\s & -c\end{array}\right) \left(\begin{array}{c}\bar{u} \\W\end{array}\right)$$
 where $s=\sin(\pi\eta)$ and $c=\cos(\pi\eta)$.
 \end{theorem}
 
 This is to say that, after adjusting for the mass distribution, the no-slip billiard collision map can be recovered from the map
 describing  the rolling around the edge  of $P$ in the limit as the radius of the  ball goes to $0$.  The required change of mass distributions amounts 
  to imposing on the  inertia parameters
 the identity
 $$\frac{\gamma^r}{\sqrt{1+\left(\gamma^r\right)^2}} 
 =\frac1\pi \arccos\left(\frac{1-\left(\gamma^b\right)^2}{1+\left(\gamma^b\right)^2}\right).$$
As an example, the uniform distribution in dimension $k+1$ corresponds to 
$\gamma^b= \sqrt{2/(k+3)}.$
 
\section{Rolling flows compared to geodesic flows}\label{flows and volume invariance}
The equations of motion of a rolling ball on a submanifold $P$ with boundary of Euclidean space  will be reviewed 
in Section \ref{derivation of newton equation}. It will be shown there that they can be given a particularly suitable form for the sake of comparison 
with dynamical systems of differential geometric origin such as  geodesic flows, magnetic geodesic  flows, and frame flows. 
In dimension $3$, they   are the Equations (\ref{roll eq dim 3}) shown above in the introduction.
We define  in this section a class of equations 
generalizing those of the rolling ball on submanifolds of Euclidean space with the intent of highlighting  the essential features of these   concrete systems.

Let $\mathcal{N}$ be an oriented  Riemannian manifold without boundary, not necessarily imbedded in Euclidean space. We fix on $\mathcal{N}$ a $(1,1)$ tensor field $\mathbb{S}$
such that $\mathbb{S}_x:T_x\mathcal{N}\rightarrow T_x\mathcal{N}$ is symmetric at all $x\in \mathcal{N}$.  (For hypersurfaces in Euclidean space, the
shape operator is the example of $\mathbb{S}$ of primary interest.) We denote the vector space of  skew-symmetric linear maps from $T_x\mathcal{N}$ to itself  as $\mathfrak{so}_x(\mathcal{N})$, whose elements will often be written as $\mathcal{S}$. 
 The vector space $\mathfrak{so}_x(\mathcal{N})$   is a fiber of the vector bundle
of skew-symmetric maps, which we denote by  $\mathfrak{so}(\mathcal{N}).$  We can now define the vector bundle
$\pi:\mathcal{M}=T\mathcal{N}\oplus \mathfrak{so}(\mathcal{N})\rightarrow \mathcal{N}.$
In rolling systems, elements $(x,u,\mathcal{S})$ of $\mathcal{M}$ represent the state of the rolling ball whose center is at $x\in \mathcal{N}$, having
center velocity $u$ and tangential spin $\mathcal{S}$.

The vector bundle $\mathcal{M}$ is given the following  Riemannian metric.
 If $e_i=(u_i,\mathcal{S}_i)$,  $i=1,2$, lie in the fiber above $x\in \mathcal{N}$,   then (up to an overall multiplicative constant)
$$\langle e_1, e_2\rangle = u_1\cdot u_2  + \frac12 \text{Tr}\left(\mathcal{S}_1\mathcal{S}_2^\intercal\right).$$
 Let $\nabla$ denote the connection on   $\mathcal{M}$ naturally induced from the Levi-Civita connection on $\mathcal{N}$. It can be shown that this is a  metric connection on the Riemannian vector bundle $\mathcal{M}$. It is worth noting that $\mathcal{M}$ may be regarded in a natural way as the tangent bundle of a certain Riemannian manifold, $\mathcal{M}=TM$,   and the connection $\nabla$ on $\mathcal{M}$
 is Riemannian on $M$. It is not clear to us at this moment whether $\nabla$ is the Levi-Civita connection (that is, torsion-free) on $M$. 
 This point is directly related to the question whether the canonical volume on $\mathcal{M}$ is invariant under rolling flow defined below in this section.

\begin{definition}[Generalized rolling equations]\label{GRE} With  the notations just defined,
we introduce the bundle map $f:\mathcal{M}\rightarrow \mathcal{M}$ such that    
$$f(e) = -\eta \left(\mathcal{S}\mathbb{S}_x u, u\wedge \mathbb{S}_xu\right),$$
where $e=(u,\mathcal{S})$. 
The {\em generalized rolling equation} is then   $\frac{\nabla e}{dt} = f(e)$. In terms of the components $u$ and $\mathcal{S}$, this is equivalent to the system
$$\dot{x} = u, \ \ \frac{\nabla u}{dt} = -\eta \mathcal{S}\mathbb{S}_x u, \ \ \frac{\nabla \mathcal{S}}{dt} = \eta (\mathbb{S}_xu) \wedge u. $$
When $\mathcal{N}$ has dimension $2$, we may write $\mathcal{S}=\mathcal{s}J_x$, where $J$ is the complex structure on $\mathcal{N}$ (that is, the
orthogonal transformation of positive rotation by $\pi/2$ for the given orientation of $\mathcal{N}$). The rolling equations then reduce to Equations (\ref{roll eq dim 3}).
 
\end{definition}

The derivation of the equations of motion of the rolling ball (Newton's equations) will be reviewed in Section \ref{derivation of newton equation}. Their derivation
is based on the Lagrange-d'Alembert principle as described in \cite{bloch}.
In the following theorem,  $P$ is a submanifold of Euclidean space of arbitrary codimension and piecewise smooth boundary. The only requirement on 
$P$ is that, for a given $r>0$, the corresponding hypersurface $\mathcal{N}=\mathcal{N}(r)$ of points at a distance $r$ from $P$ has a well defined 
shape operator at each of its points. 
\begin{theorem}[Newton's equation for the rolling ball system]\label{Newton}
When $\mathcal{N}$ is the hypersurface associated to the rolling ball on a submanifold $P$ of Euclidean space,  $\mathbb{S}$ is the shape operator on $\mathcal{N}$, and we choose to represent the state of the system   in terms of the triple $(x,u,\mathcal{S})$ where
  $x\in \mathcal{N}$ is the ball's center, $u\in T_x\mathcal{N}$ the center velocity, and $\mathcal{S}\in \mathfrak{so}_x(\mathcal{N})$
 is  the tangential spin, 
   then the rolling motion 
satisfies the generalized rolling equations on $\mathcal{M}$ given in Definition \ref{GRE}.
\end{theorem}

We make next a few remarks concerning what we refer to as the {\em rolling flow}. 
This  will be a flow on  $\mathcal{M}$.   Before defining it, it may be useful to briefly  compare the present situation with the case of
geodesic flows. For geodesic flows, $\mathcal{M}$ would correspond to the tangent bundle of a Riemannian manifold $\mathcal{N}$.
In the present case, we have in addition to the velocity $u=\dot{x} \in T_x\mathcal{N}$  the tangential spin velocity $\mathcal{S}$, which is coupled to $u$ with a coupling parameter $\eta$. When $\eta=0$, the motion on $T\mathcal{N}$ is geodesic flow, independent of the tangential  spin, while  $\mathcal{S}$ is parallel transported along geodesics. This is similar to so-called {\em orthogonal frame flows},  except that in our case what is parallel transported (for $\eta=0$) is not an orthonormal frame but a tensor that is related to the state of  spinning  of a frame relative to itself.

An immediate feature of the force term $f$ in Newton's equation is that $\langle e, f(e)\rangle=0$.
This follows from the easily derived identity
  $$\frac12\text{Tr}\left(\mathcal{S}\left(\mathbb{S}u \wedge u\right)^\intercal\right)= \langle u, \mathcal{S}\mathbb{S} u\rangle$$
  and the definition of the Riemannian metric.
 In particular, solution curves $e(t)$ have constant energy:
$\mathcal{E}(e(t))=\mathcal{E}(e(0))$, where
$$\mathcal{E}(e) =\frac12\|e\|^2 =  \frac12\left(|u|^2 + \frac12 \text{Tr}\left(\mathcal{S}\mathcal{S}^\intercal\right)\right). $$

The connection induces a splitting $T\mathcal{M}=E^V\oplus E^H$ as a direct sum into vertical and horizontal subbundles, and
a {\em connection map} $K_e:T_e\mathcal{M} \rightarrow \mathcal{M}_x$ (here $\mathcal{M}_x$, $x=\pi(e)$, is the vector
fiber of $\mathcal{M}$ at $x\in \mathcal{N}$.)  Recall that if $e(t)\in \mathcal{M}$ is a differentiable curve so that $e(0)=e$ and $e'(0)=\xi$, then
$$ K_e\xi = \frac{\nabla e}{dt}(0).$$
Let $Z$ be the horizontal vector field on $\mathcal{M}$ such that $d\pi_eZ(e)=e$ for all $e\in \mathcal{M}$, which we may call the {\em extended geodesic spray}, and let $F$ be the vertical vector field 
such that $K_e F(e) = \frac1\eta f(e).$ Then define $V=Z+\eta F$. 

\begin{definition}[Rolling flow]
The {\em rolling flow} is the flow on $\mathcal{M}$   whose infinitesimal generator is the vector field  $V$. 
\end{definition}

\begin{proposition} A curve
 $t\mapsto e(t)$ in $\mathcal{M}$ is an integral curve of $V=Z+\eta F$ if and only if, regarded as a vector field along $x(t)=\pi(e(t))$, it satisfies Newton's equation of Definition \ref{GRE}. 
  Moreover, the rolling flow leaves invariant the kinetic energy function $\mathcal{E}$.  When $\eta=0$, the  flow line with initial condition $(x,u,\mathcal{S})$ corresponds to the parallel transport of $\mathcal{S}$ along the geodesic
  in $\mathcal{N}$ with initial condition $(x,u)$.
\end{proposition}

 \section{Elementary examples} \label{elementary examples}
 A more detailed study of the generalized rolling equations and their dynamical properties  will be taken up elsewhere. Here
 we consider  a few elementary examples for which the rolling equation can in principle  be solved analytically. The case of rolling
 over $\mathbb{R}^k$ in $\mathbb{R}^m$, $k<m$,  is of special interest. These naturally arise in the context of rolling on polygonal and polyhedral shapes, where one needs to account for motion over polyhedral  faces of various  dimensions. 
 
 Thus let $P=\mathbb{R}^k$, regarded as a submanifold of $\mathbb{R}^m$. The locus of centers of the rolling ball of radius $r$ is then $\mathcal{N} =  \mathbb{R}^k\times S^{m-k-1}(r) $, where $S^\ell(r)$ is the sphere of radius $r$ in 
 $\mathbb{R}^{\ell+1}$ centered at the origin.  At any given $x\in \mathcal{N}$ let $\Pi_x:T_x\mathcal{N}\rightarrow \mathbb{R}^k$ denote the
 orthogonal projection, where we identify   $\mathbb{R}^k$ with  its tangent space at any given point.  We also write $\Pi_x^\perp=I-\Pi_x$. 
 Now express the center of mass velocity $u$ and tangential angular velocity operator $\mathcal{S}$ at any given point $x$ as $u=u_0+u_1$ and
 $\mathcal{S}=\mathcal{S}_{00}+\mathcal{S}_{01}+\mathcal{S}_{10}+\mathcal{S}_{11}$ where
 $$u_0=\Pi_x u, \ u_1=\Pi_x^\perp u, \ \mathcal{S}_{00}=\Pi_x\mathcal{S}\Pi_x, \  \mathcal{S}_{01}=\Pi_x\mathcal{S}\Pi^\perp_x, \ 
 \mathcal{S}_{10}=\Pi^\perp_x\mathcal{S}\Pi_x, \ \mathcal{S}_{11}=\Pi^\perp_x\mathcal{S}\Pi^\perp_x.$$
 Note that the tensors $\Pi$ and $\Pi^\perp$ are parallel and the shape operator satisfies $\mathbb{S}_x  = -\frac1r \Pi^\perp_x$. The rolling  equations of   
 Definition \ref{GRE} can be shown without difficulty to take the form of the system of equations
\begin{equation}\label{Rk in Rm}\dot{u}_0 = \frac{\eta}{r} \mathcal{S}_{01}u_1, \  \dot{\mathcal{S}}_{00}=0, \ \frac{\nabla u_1}{dt} = \frac{\eta}{r}\mathcal{S}_{11} u_1,  \ \frac{\nabla \mathcal{S}_{11}}{dt}=0, \ \frac{\nabla\mathcal{S}_{01}}{dt} = -\frac{\eta}{r} u_1^\flat \otimes u_0.\end{equation}
Here $u_1^\flat$ is the covector dual to $u_1$, so that $(u_1^\flat\otimes u_0)v = (u_1\cdot v) u_0$.
At the core of the effort  of solving  these equations is the rolling problem on spheres (the third and fourth  equations). Let us consider the cases of codimensions $1, 2$ and $3$.

    \begin{example}[Codimension $1$]\label{cod1}\em 
Here $\mathcal{N}$ consists of two parallel planes of dimension $k$ in $\mathbb{R}^{k+1}$ (a distance $2r$ apart) and $u_0=u, \mathcal{S}_{00}=\mathcal{S}$.
Equations \ref{Rk in Rm} reduce to $\dot{u}=0$ and   $\dot{\mathcal{S}}=0$. Thus the ball rolls with constant center velocity and constant 
tangential spin $\mathcal{S}$. In dimension $m=3$, the latter means that the normal component $\mathcal{s}$ of the tangential spin is constant. 
  \end{example}

  \begin{example}[Codimension $2$]\label{cod2}\em 
  In this case we have $\mathcal{N}= \mathbb{R}^k\times S^{1}(r)$.  This Riemannian manifold admits a parallel orthonormal frame of vector fields $E_1, \dots, E_k, E$,
  where the $E_i$ are tangent to $\mathbb{R}^k$ and $E$ is tangent to the circle.   We have $\mathcal{S}_{11}=0$ since it is skew-symmetric and rank $1$.
  Therefore $\mu:=u\cdot E = u_1\cdot E$ is constant. This means that trajectories rotate around the circle factor at a constant rate $\mu$. 
  Let us define the quantity $w:=\mathcal{S}_{01}E$. Then $w$ and $u_0$ are both vectors in $\mathbb{R}^k$ and are related by the system of linear        equations
  $$ \dot{u}_0 = \frac{\eta \mu}{r} w, \ \ \dot{w} = -\frac{\eta\mu}{r} u_0.$$
Rolling trajectories project to ellipses  in $\mathbb{R}^k$ with the following parametric equation: 
$$ x(t) =  \cos(\omega t)\mathbbm{a} + \sin(\omega t)\mathbbm{b}+\mathbbm{c}$$ 
where $\omega = \eta \mu /r$ and 
  $$ \mathbbm{a}=  w(0)/\omega, \ \ \mathbbm{b} = u_0(0)/\omega, \ \  \mathbbm{c} = x(0) - (1/\omega)^2 \dot{u}_0(0)=x(0)-(1/\omega)w(0).$$  The quantities $(\mathcal{S}E_i)\cdot E_j$, $1\leq i<j\leq k$ are constants of motion. 
    \end{example}

\begin{figure}[htbp]
\begin{center}
\includegraphics[width=2in]{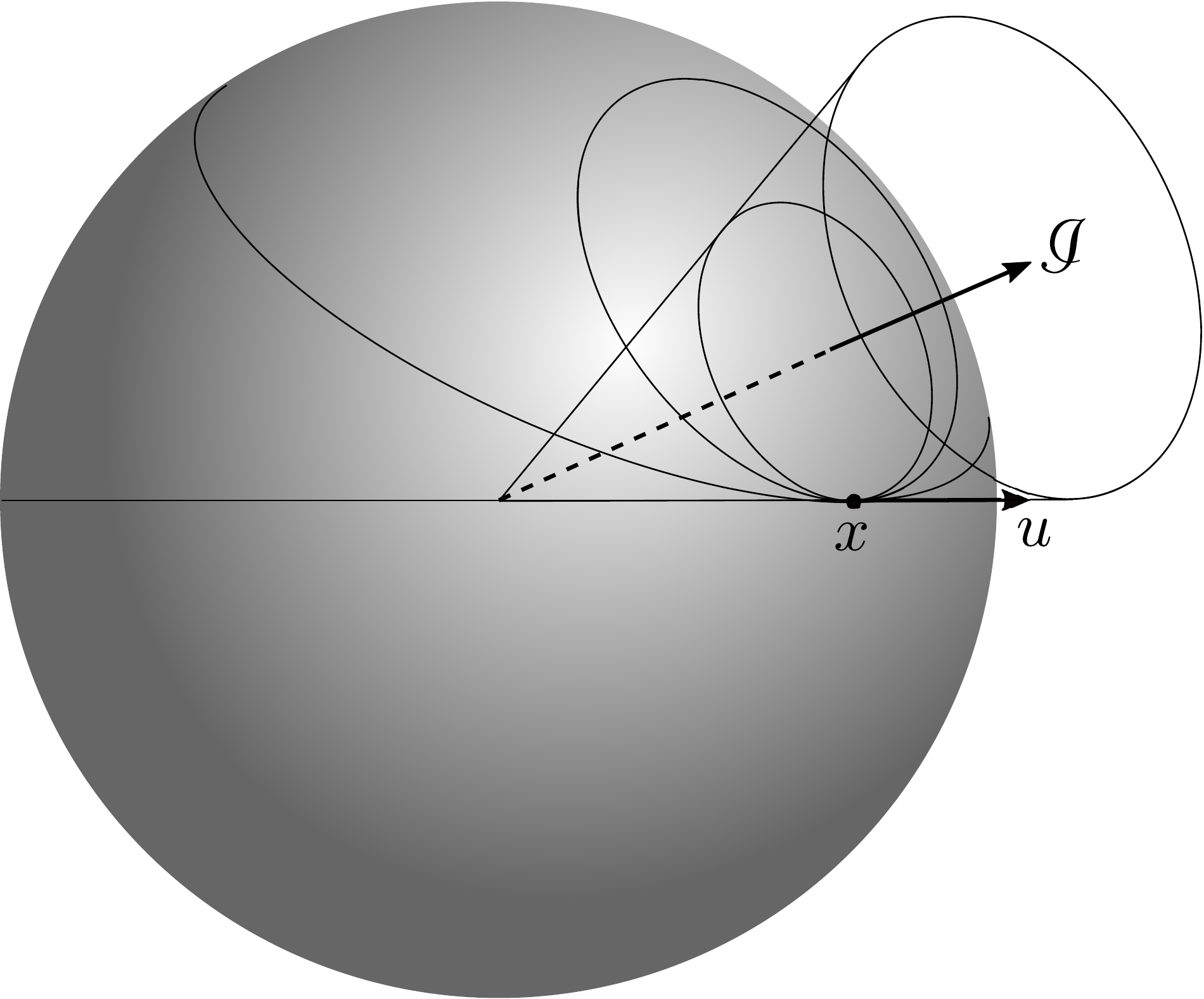}\ \ 
\caption{{\small  Rolling trajectories on spheres in $\mathbb{R}^3$ are the circles of intersection of the sphere with  cones.  See Example \ref{example cod 3}.}}
\label{cones}
\end{center}
\end{figure}

  \begin{example}[Codimension $3$]\label{example cod 3}\em
  In this case $\mathcal{N}= \mathbb{R}^k\times S^2(r)$. Points of $\mathcal{N}$ will be written $x=(x_0,x_1)$ where $x_0$ is the component in $\mathbb{R}^k$ and $x_1$ the component on the sphere. Let $J_{x_1}:T_{x_1} S^2(r)\rightarrow T_{x_1}S^2(r)$ denote positive  rotation   by $\pi/2$ (taking
  the outward pointing normal vector for the orientation of the sphere.) The tensor $J$ is parallel. 
Then  $\mathcal{S}_{11} = \mathcal{s} J$, $\frac{\nabla \mathcal{S}_{11}}{dt} =\dot{\mathcal{s}} J$ and
 the fourth among Equations (\ref{Rk in Rm}) implies that $\mathcal{s}$ is a constant of motion.  
 The third equation then turns into a linear equation on the sphere:
 \begin{equation}\label{roll on sphere} \frac{\nabla u_1}{dt} = \frac{\eta\mathcal{s}}{r} Ju_1.\end{equation}
An immediate consequence   is that $|u_1|^2$ is constant. Now define the quantity
$$\mathcal{I}:= \eta\mathcal{s} x_1 +x_1\times u_1,$$
which is a vector in $\mathbb{R}^3$ (the orthogonal complement to $\mathbb{R}^k$ in $\mathbb{R}^{k+3}$.) Here $\times$ is
the standard cross-product. Observe that
$$\dot{\mathcal{I}} = \eta\mathcal{s} u_1 +x_1\times \frac{\nabla u_1}{dt} =\eta\mathcal{s}\left(u_1 + \frac{x_1}{r}\times Ju_1\right) =0 $$
since $u_1, Ju_1$ and $x_1/r$ form a positive orthonormal basis of $\mathbb{R}^3$. Thus $\mathcal{I}$ is a constant of motion (only depending on the initial conditions).  Also note that $\mathcal{I}\cdot x_1=\eta\mathcal{s} r^2$ is constant. This means that the projection to $S^2(r)$ of
rolling trajectories are circles, traversed with uniform speed, given by the intersection of $S^2(r)$ and level sets of the function $x_1\mapsto \mathcal{I}\cdot x_1$, which are cones. (See Figure \ref{cones}.)
Let $w_1:=\mathcal{S}_{01} u_1$ and $w_2:=\mathcal{S}_{01} Ju_1$. Then $w_1, w_2, u_0\in \mathbb{R}^k$  are related by the linear system
$$\dot{w}_1 = -\frac{\eta |u_1|^2}{r} u_0 +\frac{\eta \mathcal{s}}{r} w_2, \ \dot{w}_2 = -\frac{\eta}{r} w_1, \ \dot{u}_0=\frac{\eta}{r}w_1.$$
It is now a simple calculation to solve for $w_1, w_2$ and $u_0$, as well as $\mathcal{S}_{01}$. The projection to $\mathbb{R}^k$ of rolling trajectories
are ellipses.
  \end{example}

\begin{example}[Rolling around a straight edge]  \label{around edge}\em We can use the analysis of Example \ref{cod2} to obtain a billiard interpretation 
of the rolling around the edge of a half-space.
Let $P=\mathbb{R}^{k-1}\times [0,\infty)$ be the half-space in $\mathbb{R}^k$. We view $P$ as the submanifold of $\mathbb{R}^{k+1}$ consisting of
points $x=(x_1, \dots, x_{k+1})$  such that $x_k\geq 0$ and $x_{k+1}=0$. The manifold boundary of $P$ is the subspace $\mathbb{R}^{k-1}$ corresponding to
$x_k=x_{k+1}=0$. Let   $\mathbbm{n}=(0, \dots, 0, 1,0)$ be
the unit normal vector to  $\partial P$ pointing into $P$.
Then $\mathcal{N}$ (defined for a radius $r>0$) is the product $\mathcal{N}=\mathbb{R}^{k-2}\times \mathcal{R}$ where $\mathcal{R}$ is the piecewise smooth curve in $\mathbb{R}^2$ depicted on the left-hand side of  Figure \ref{transversal}. 
\begin{figure}[htbp]
\begin{center}
\includegraphics[width=3in]{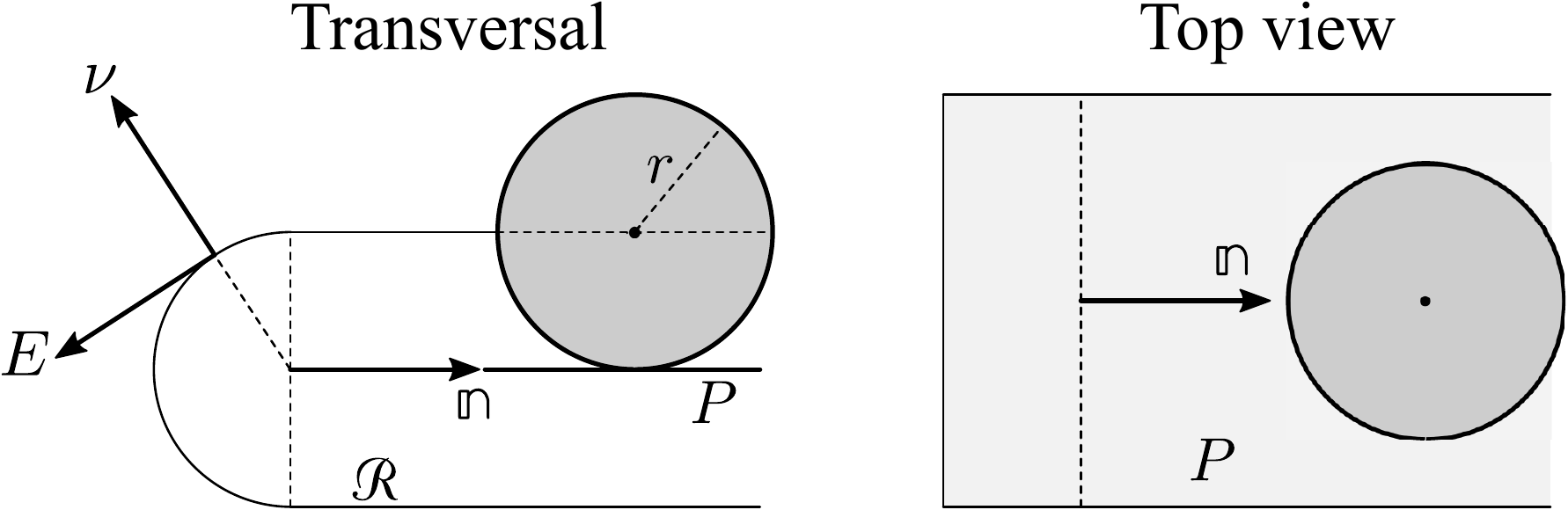}\ \ 
\caption{{\small  Notation for  Example \ref{around edge}.}}
\label{transversal}
\end{center}
\end{figure}  
We wish to determine the map that gives the velocities of the rolling ball after rolling around the edge as a function of the velocities it had immediately before.  Rolling around the edge itself is described in Example \ref{cod2}. Let $E_1, \dots, E_{k-1}, E$ be as in that example. 
We know that $\mu=u\cdot E$ is constant, so the time it takes the ball to roll all the way around the edge (from the moment it leaves, say, the top flat sheet of $\mathcal{N}$ to the moment it enters the bottom one or vice versa) is $T=\pi r /|\mu|.$ The tangential angular velocity  matrix $\mathcal{S}$ is fully specified by the components
$w_i=(\mathcal{S}E)\cdot E_i$, $i=1, \dots, k-1$, and the constants $(\mathcal{S}E_i)\cdot E_j$, $1\leq i<j\leq k-1$. 
The quantities
  $u_0=\sum_{i=1}^{k-1}(u\cdot E_i) E_i$  and $w=\sum_{i=1}^{k-1}w_i E_i$ satisfy, on the curved part of $\mathcal{N}$, the system of differential equations
 of Example \ref{cod2} whose solution can be written as follows: 
 $$\left(\begin{array}{c}u^+_0 \\w^+\end{array}\right) =\exp\left\{ \frac{\eta\mu}{r}T\left(\begin{array}{rr}0 & I \\-I & 0\end{array}\right)\right\}\left(\begin{array}{c}u^-_0 \\w^-\end{array}\right) =     \left(\begin{array}{rr}\cos(\sigma\pi\eta)I & -\sin(\sigma\pi\eta)I \\\sin(\sigma\pi\eta)I & \cos(\sigma\pi\eta)I\end{array}\right)   \left(\begin{array}{c}u^-_0 \\w^-\end{array}\right).$$
 Here $-$ and $+$ indicate the velocities before and after rolling around the curved part of $\mathcal{N}$
 and $\sigma\in \{+,-\}$ is the sign of $\mu$. This sign is positive when rolling begins at the top sheet of $\mathcal{N}$ and negative otherwise. 
   For our later needs, we rewrite this relation
 as follows. Define $W^\pm= \mathcal{S}\mathbbm{n}$. (Here, as for $w^\pm$, the sign indicates ``before'' and ``after.'') With these conventions
 we have
\begin{equation}\label{billiard interpretation} \left(\begin{array}{c}u^+_0 \\W^+\end{array}\right)  =     \left(\begin{array}{rr}\cos(\pi\eta)I & \sin(\pi\eta)I \\\sin(\pi\eta)I & -\cos(\pi\eta)I\end{array}\right)   \left(\begin{array}{c}u^-_0 \\W^-\end{array}\right). \end{equation}
 Relation (\ref{billiard interpretation}) has the following billiard interpretation. If one observes from above (see the right-hand side diagram of Figure  \ref{transversal}) the motion of a small ball  that rolls around the edge of the half-space, the motion appears as a collision of a rotating disc (the flattened ball) with the boundary of $P$;
 the   component of the velocity of the center of the disc perpendicular to the boundary of $P$  changes sign while the other components of this velocity and
 the components of the tangential  spin matrix $\mathcal{S}$ are exchanged according to Equation (\ref{billiard interpretation}).
 In particular,  if $\eta=0$, the disc undergoes ordinary billiard (specular) reflection and its direction of rotation   is reversed. Also note that the linear map in (\ref{billiard interpretation}) does not depend on the radius $r$. 
\end{example}

\begin{figure}[htbp]
\begin{center}
\includegraphics[width=3in]{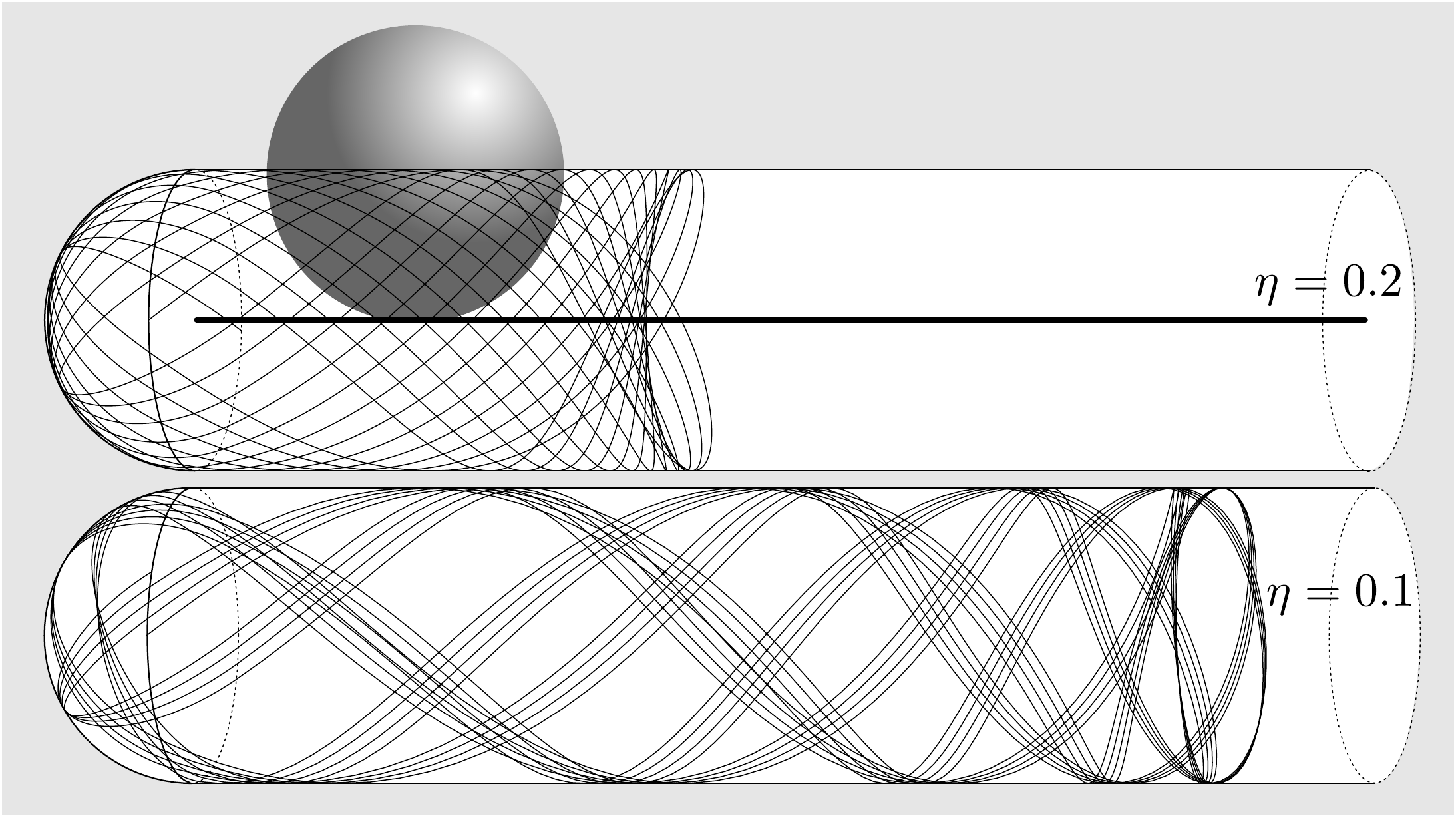}\ \ 
\caption{{\small  The elementary examples  \ref{cod1}, \ref{cod2}, \ref{example cod 3} are building blocks for rolling on polygonal or polyhedral plates. 
As a
simple illustration in dimension $1$,  rolling on a semi-infinite line in $\mathbb{R}^3$ is described by putting together the codimension $2$ and codimension $3$
examples.}}
\label{semi_infinite}
\end{center}
\end{figure}

\begin{example}[Rolling on a semi-infinite line in $\mathbb{R}^3$]\em   
The above and similar examples in higher codimension   may be regarded as the building blocks of rolling systems on polygonal or polyhedral convex shapes. For example, rolling on a convex polygonal plate in dimension $3$ involves rolling on the  surface interior (codimension $1$), on the edges (codimension $2$) and on the vertices (codimension $3$).
As a very simple example of this idea, consider the rolling of a ball in dimension $3$ over a semi-infinite straight line. (See Figure \ref{semi_infinite};  we plan a detailed investigation  of
the dynamics of rolling flows on polygonal  plates in a future study.)  This makes use of the codimensions $2$ and $3$ examples. 
The following observations use notation from Example \ref{cod2}. 
The maximum displacement along the semi-infinite line as a function of initial conditions and parameters is 
$$ \mathcal{m}_d=\frac{1}{\omega}\left[\sqrt{(u_0(0))^2+(\mathcal{s}(0))^2} - \mathcal{s}(0)\right].$$
We recall that  $\omega= \eta\mu/r$, where $\mu$ is the constant  velocity of rotation about the axis of the cylinder.
Here the initial point is on the equator of the spherical cap and the initial velocity points into the cylindrical end. 
In particular, for a given set of initial conditions, this displacement is a linear function of $\eta^{-1}$. Naturally, as the moment of inertia  parameter $\eta$ approaches $0$ and the
trajectory of the center of the ball approaches a geodesic path, this displacement, generically,  approaches infinity. 
In each excursion from, and back to, the equator of the spherical cap, the initial and final values of $\mathcal{s}$ are the same ($\mathcal{s}$ is constant on the spherical cap but not on the cylinder) and $u_0$ changes sign, while the projection of $u$ to the plane orthogonal to the axis of the cylinder simply rotates.

\end{example}

\section{Review of differential geometry of rolling} 
\label{derivation of newton equation}
For the sake of completeness and in order to establish notation, we review in this section  the derivation of   the equations of motion of a ball of
dimension $m$ with spherically symmetric mass distribution rolling without slipping on a submanifold $P$ of $\mathbb{R}^m$. The submanifold is allowed to have nonempty  boundary (or corners) and arbitrary codimension although certain  natural restrictions on corners are needed in order that the hypersurface $\mathcal{N}$ be differentiable.   (As a typical example having  nondifferenciable $\mathcal{N}$, consider the flat plate $P$ in $\mathbb{R}^3$ given by the set difference of $\mathbb{R}^2$ minus the first quadrant.)
\subsection{Preliminaries on constrained rigid motion}
 We begin by laying out some preliminary information about the Euclidean group.
Let ${SE}(m)={SO}(m)\ltimes \mathbb{R}^m$ denote the special Euclidean group of orientation preserving isometries of Euclidean space.   In this section only, we find it convenient to  write elements of $SE(m)$ as $g=(a,A)$ (rather than $(x,A)$ as before); they operate  on $\mathbb{R}^m$ by rigid transformations according to the action $g(x)=Ax+a$. We regard $SE(m)$ as the
configuration manifold of a {\em rigid body} $\mathcal{B}\subset \mathbb{R}^m$. Here $\mathcal{B}$ is, for the moment, a general  measurable set with 
mass distribution defined by a finite positive measure $\mu$, but will be shortly specialized to a ball of radius $r$.  A {\em motion} of $\mathcal{B}$ is a  differentiable  path $g(t)=(a(t), A(t))\in SE(m)$. 
We write $g=g(0)$, $\xi=g'(0)=(a',A')$, $u_\xi=a'$, and $U_\xi= A'A^{-1}\in \mathfrak{so}(m)$. The velocity of material point $x\in \mathcal{B}$ at $t=0$ is
$$ V_x(g,\xi)=\left.\frac{d}{dt}\left(A(t)x+a(t)\right)\right|_{t=0}= A'x+a' = U_\xi Ax+u_\xi.$$
We refer to $U_\xi$ as the {\em angular velocity} matrix.  See Figure \ref{state}.

\begin{figure}[htbp]
\begin{center}
\includegraphics[width=4.0in]{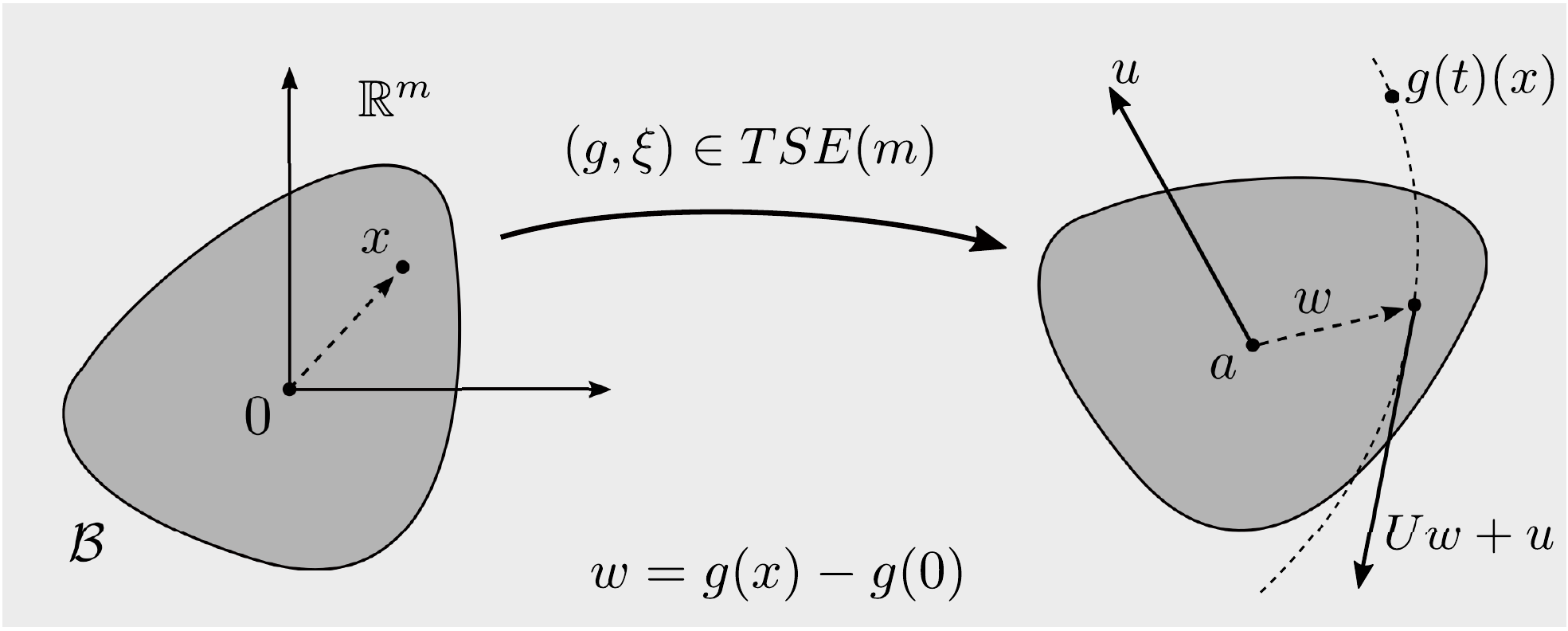}\ \ 
\caption{{\small The state $(g,\xi)\in TSE(m)$  gives  the velocity $V_x(g,\xi)=Uw+u$ of each material point $x\in \mathcal{B}$ at a moment in time.  The notation $g(t)(x)$ means  the transformation of $x$ due to configuration $g(t)\in SE(m)$.}}
\label{state}
\end{center}
\end{figure}

The kinetic energy of the moving body is the  function  of the state $\xi\in T_gSE(m)$ at configuration $g$ given by
$$K_g(\xi)=\frac12\int_{\mathcal{B}} |V_x(g,\xi)|^2\, d\mu(x). $$
 It is the quadratic form associated to the symmetric bilinear form 
$$ \langle\xi,\zeta\rangle_g:= \mathcal{m} \left[ u_\xi\cdot u_\zeta + \frac12 \text{Tr}\left(\mathcal{L}(U_\xi) U_\zeta^\intercal\right)\right].$$
In this expression, $\mathcal{m}=\mu(\mathcal{B})$ is the total mass of the body, $ u_\xi\cdot u_\zeta$ is the ordinary dot product, and $\mathcal{L}(U)$  is a certain linear map on $\mathfrak{so}(m)$ that depends on the second moments of the mass distribution $\mu$, as defined in \cite{CF}. When $\mathcal{B}$  is  a ball of radius $r$ centered at the origin of $\mathbb{R}^m$ and $\mu$ is a rotationally symmetric mass distribution,  $\mathcal{L}$ is  a scalar transformation. The resulting bilinear form in this case defines the following Riemannian metric on $SE(m)$: 
\begin{equation}\label{Riemannian metric} \langle\xi,\zeta\rangle_g:= \mathcal{m} \left[ u_\xi\cdot u_\zeta + \frac{r^2\gamma^2}2 \text{Tr}\left(U_\xi U_\zeta^\intercal\right)\right].\end{equation}
where $\gamma$ is a {\em moment of inertia} parameter. We also define $\eta=\gamma/\sqrt{1+\gamma^2}\in [0,1)$. When $\eta=\gamma=0$, the entire mass of the body is concentrated at the center of the ball. In this case body rotation does not contribute to the kinetic energy and the inner product becomes degenerate.

\begin{figure}[htbp]
\begin{center}
\includegraphics[width=4.0in]{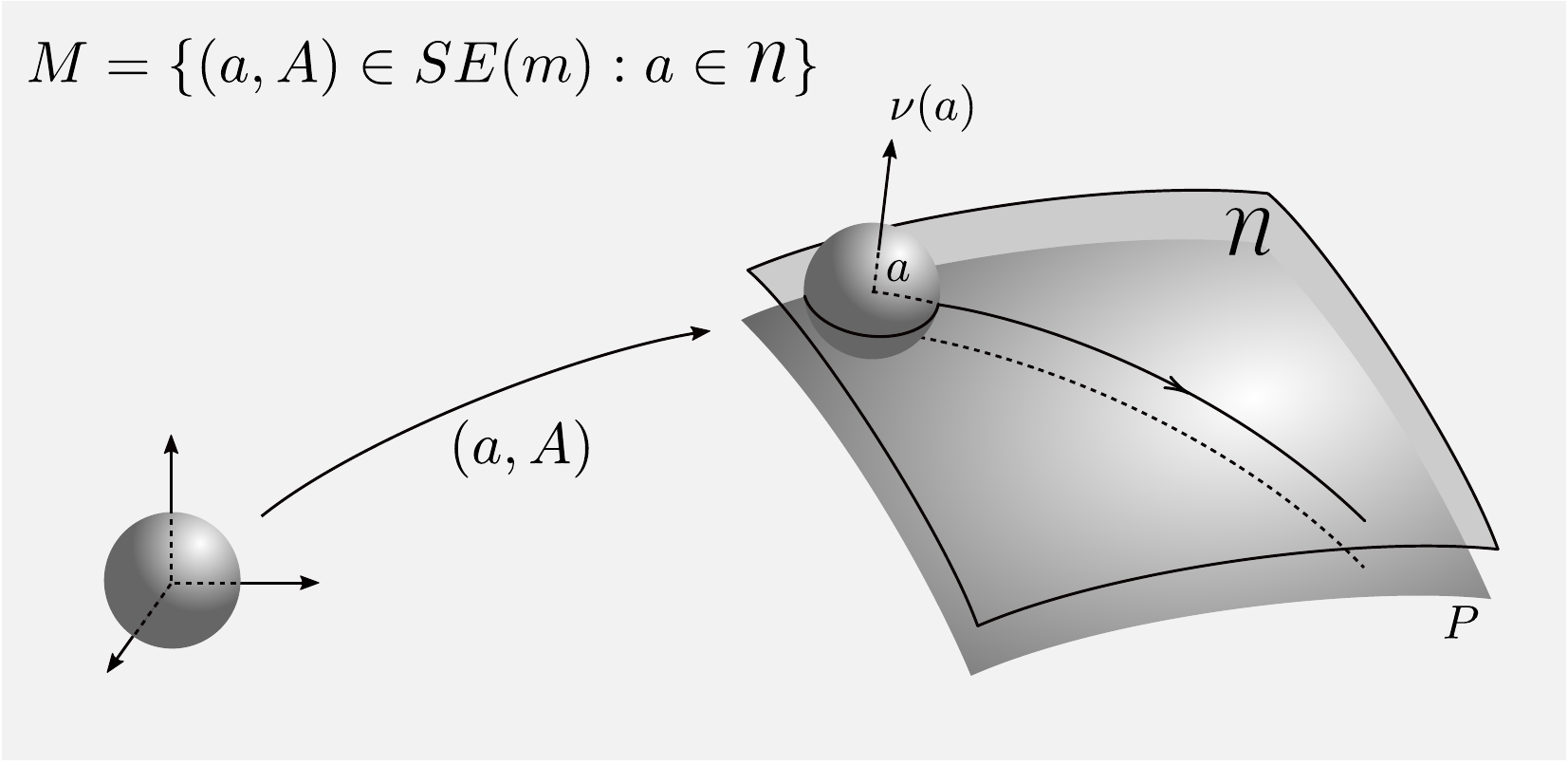}\ \ 
\caption{{\small   The rolling ball satisfies the holonomic  constraint of being in tangential contact with the submanifold $P$, hence its center should
lie on the hypersurface $\mathcal{N}$ at distance $r$ (the radius of the ball) from $P$. It also satisfies the non-holonomic constraint that the velocity at the
point of contact with $P$ is zero.  }}
\label{roll}
\end{center}
\end{figure}

 Let us now restrict  the motion   so that the ball rolls over a submanifold  $P$ of $\mathbb{R}^m$  without slipping. More details can be found in \cite{CCCF}.
 The set-up is as shown in Figures \ref{roll} and \ref{constraint}. The locus of possible centers  is  $\mathcal{N}=\mathcal{N}(r)$. Thus $\mathcal{N}$ is the set of points in $\mathbb{R}^m$ at distance $r$ from $P$. We assume that $P$ is such that   $\mathcal{N}$ is an imbedded  submanifold of $\mathbb{R}^m$ for  sufficiently small $r$. Note that, when $P$ has boundary, $\mathcal{N}$ may fail to be smooth even if $P$ is smooth  although the unit normal vector field $a\mapsto \nu(a)$ to $\mathcal{N}$ (pointing to the side of rolling)  often is, and is here assumed to be, piecewise smooth and continuous.
 The {\em no-slip} constraint requires   the velocity of the point on the ball in configuration $g$ and   in tangential contact with $P$ to  be zero.

This is expressed analytically as follows.
At  state $(g,\xi)$, the velocity of the contact point  $p=g(x)$  is $$V_x(g,\xi)= u_\xi - r U_\xi \nu(a)$$ where $g=(a,A).$
    Note that $V_x(g,\xi)$ is the sum of the velocity $u_\xi$ of the center  $a$ of the ball in configuration $(a,A)$  and the velocity $-rU_\xi \nu(a)$ of the contact point   $p$ relative to  $a$.
  Thus the constraint equation is $$u_\xi = rU_\xi \nu(a).$$ See Figure \ref{constraint}.
This equation defines a vector subbundle of $TM$, which we call the {\em rolling bundle} and denote $\mathfrak{R}$. Thus we require of the motion $t\mapsto g(t)$ that $g(t)\in M$ and $g'(t)\in \mathfrak{R}_{g(t)}$ for all $t$.

\begin{figure}[htbp]
\begin{center}
\includegraphics[width=2.5in]{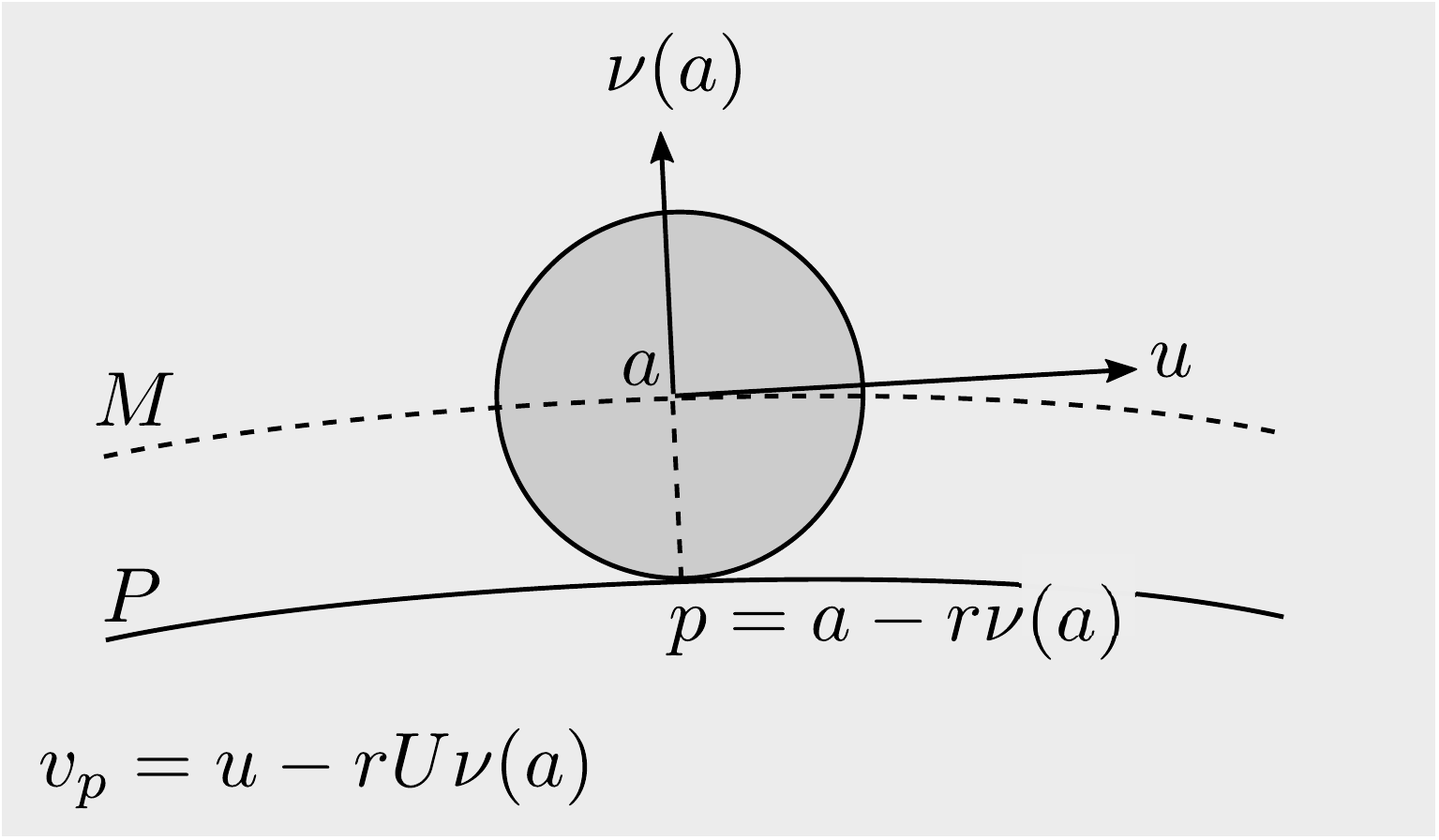}\ \ 
\caption{{\small   The velocity $v_p$ of the point of contact is zero under the no-slip constraint.}}
\label{constraint}
\end{center}
\end{figure}

 \begin{definition}[Rolling bundle]
 The {\em rolling bundle} is the vector subbundle $\mathfrak{R}$ of $TM$, where $M=\{g=(a,A)\in SE(m): a\in \mathcal{N}\}$, such that
 $$\mathfrak{R}_g = \{(u, UA)\in T_gM: u = rU\nu(a)\}. $$
 \end{definition}

 \begin{proposition}\label{R_perp prop}
 At $g=(a,A)\in M$ let $\mathfrak{R}^\perp_g$ denote the orthogonal complement of $\mathfrak{R}_g$ in $T_gSE(m)$ with respect to the kinetic energy Riemannian metric. Then 
 \begin{equation}\label{R_perp}\mathfrak{R}_g^\perp =\left\{\left(\frac{1}{r\gamma^2} w \wedge \nu(a) A, w\right): w\in \mathbb{R}^m\right\}\end{equation}
 and $\mathfrak{R}_g^\perp\bigcap T_gM$ has the same expression  but with $w\in T_a\mathcal{N}$. It follows that $\dim \mathfrak{R}^\perp_g=m$, $\dim \mathfrak{R}^\perp_g\bigcap T_gM=m-1$ and $\dim \mathfrak{R}_g=\dim SE(m)-\dim \mathfrak{R}_g^\perp=\frac{m(m-1)}{2}$.
 \end{proposition}
 \begin{proof}
The main observation, from which the rest follows, is the orthogonality between $\mathfrak{R}_g$ and the subspace defined by the right-hand side of Equation (\ref{R_perp}). This is shown by making use of the easily verified general identity
$$ \frac12 \text{Tr}\left((v\wedge w) U^\intercal\right)=\left(Uv\right)\cdot w.$$
Granted this identity and having in mind that $u=rU\nu(a)$ when $\left(u,UA\right)\in \mathfrak{R}_g$, we obtain
\begin{align*}
\left\langle \left(u,UA\right), \left(w,\frac1{r\gamma^2} w\wedge \nu(a) A\right)\right\rangle &=  \mathcal{m} \left\{ u\cdot w +\frac{r^2\gamma^2}{2}\text{Tr}\left(\frac1{r\gamma^2} {w\wedge \nu(a)} U^\intercal\right)\right\}\\
&= \mathcal{m}\left\{r w\cdot U\nu(a) + r\left(Uw\right)\cdot \nu(a)\right\}
\end{align*}
but the last expression is $0$ since $U$ is skew symmetric.
 \end{proof}

 \subsection{Newton's equation under constraint}
 We  obtain  next Newton's equations for the motion of the rolling ball. 
 The Riemannian metric given above on $SE(m)$  is a product metric that agrees with the Euclidean metric on the normal  subgroup   $\mathbb{R}^m$ and defines a bi-invariant (for the rotationally symmetric mass distribution) Riemannian metric on   $SO(m)$. The latter metric is $\langle X,Y\rangle = c\, \text{Tr}\left(XY^\intercal\right)$ where  $c$ is a positive constant and $X,Y\in \mathfrak{so}(m)$. It is easily shown that the Levi-Civita connection on $SO(m)$ satisfies
 $$\nabla_XY = \frac12[X,Y] $$
 for left-invariant vector fields $X,Y$. 
 
 We often denote the derivate of parametric curves $A(t)$ as $\dot{A}$  when the parameter is interpreted as time.
It is useful to note that if
  $A(t)\in SO(m)$ is a twice differentiable curve and  $U(t)=\dot{A}(t)A(t)^{-1}$, then
 $$\frac{\nabla \dot{A}}{dt} =\dot{U} A.$$
 In fact, let $E_1, \dots, E_N$, $N=m(m-1)/2$, be  orthonormal, right-invariant vector fields on $SO(m)$. Note that $\left\langle \dot{A}, E_j\right\rangle_A=
 \left\langle U, E_j\right\rangle_I$ so  $\dot{A}(t) = \sum_j\left\langle U,E_j\right\rangle_I E_j(A(t))$. Then
 \begin{align*}
 \frac{\nabla \dot{A}}{dt}&=\sum_j\left\langle \dot{U}, E_j\right\rangle_I E_j(A) + \sum_j \left\langle {U}, E_j\right\rangle_I \nabla_{\dot{A}}E_j(A)\\
 &=\dot{U}A + \sum_{j,k} \left\langle {U}, E_j\right\rangle_I \left\langle \dot{A}, E_k(A)\right\rangle_A \nabla_{E_k}E_j(A)\\
 &= \dot{U}A + \frac12\sum_{j,k} \left\langle {U}, E_j\right\rangle_I \left\langle {U}, E_k\right\rangle_I  [E_k,E_j](A)\\
 &=  \dot{U}A + \frac12 [U,U]\\
 &=  \dot{U}A.
 \end{align*}
 Similarly, if we had defined  $U(t)=A(t)^{-1}\dot{A}(t)$, then $\frac{\nabla \dot{A}}{dt}= A\dot{U}$. It follows that if $g(t)$ is a twice differentiable path
 in $SE(m)$, then
 $$\frac{\nabla \dot{g}}{dt} = \left( \dot{U}A, \Pi_a \dot{u}\right) $$
 where $\Pi_a:\mathbb{R}^m\rightarrow T_a\mathcal{N}$ is the orthogonal projection.

Newton's equation takes the form 
$$\mathcal{m} \frac{\nabla \dot{g}}{dt} = N(g,\dot{g})$$
where $N(g,\dot{g})\in \mathfrak{R}_g^\perp$ is the constraint force required for the condition $\dot{g}\in \mathfrak{R}_g$ to hold without producing work. Thus $N(g,\dot{g})=\left(w,\frac1{r\gamma^2} w\wedge \nu(a) A\right)$ for some $w\in \mathbb{R}^m$. Inserting the expressions for $N$ and $\nabla\dot{g}/dt$ into Newton's equation gives the first two below equalities, the third being the velocity constraint condition:
\begin{equation}\label{three_equations}
\mathcal{m}\dot{U} =  \frac1{r\gamma^2} w\wedge \nu(a), \ \ 
\mathcal{m}\Pi_a \dot{u}  =  w, \ \ 
u= rU\nu(a).  
\end{equation}
Differentiating the third equation, keeping in mind that the derivative of $\nu(a(t))$ in $t$ equals  $-\mathbb{S}_{a(t)} u$ where $\mathbb{S}_a$ is the shape operator of $\mathcal{N}$ at $a$, gives
$$ \dot{u}=r\dot{U}\nu(a) - rU\mathbb{S}_a u.$$
From this and the identity $-(w\wedge \nu(a))\nu(a)= w-w\cdot\nu(a) \nu(a) = \Pi_aw$ together with the 
first of the above three equations in (\ref{three_equations}), we obtain
$$\mathcal{m}\left(\dot{u}+rU\mathbb{S}_au\right)  =   -\frac1{\gamma^2}\Pi_aw. $$
The second of the identities in (\ref{three_equations}) now implies
$$w+ \mathcal{m}r\Pi_a U \mathbb{S}_a u = -\frac1{\gamma^2}\Pi_aw,$$
from which we obtain, after some manipulation, the value of $w$ in terms of   $U$:
\begin{equation}\label{w}
w = -\mathcal{m} r\frac{\gamma^2}{1+\gamma^2}\Pi_aU \mathbb{S}_a u.
\end{equation}
Using once more the first of the equalities in (\ref{three_equations}) we arrive at the equation
\begin{equation}\label{first equation}
\mathcal{m}\dot{U}= \frac{1}{r\left(1+\gamma^2\right)}r\mathcal{m}U\mathbb{S}_a u \wedge\nu(a).
\end{equation}
\begin{proposition}[Rolling equation I]\label{propo_roll_equation}
Let $g(t)=(a(t),A(t))$ represent the motion of the ball with rotationally symmetric mass distribution and moment of inertia parameter $\gamma$, subject to
 the nonholonomic constraint defined by the rolling distribution $\mathfrak{R}\subset TM$. Then
$$ \dot{U}= -\frac{r}{\left(1+\gamma^2\right)}\left(U\mathbb{S}_a U\nu(a)\right)\wedge\nu(a).$$
From a solution $U$ of one of these equations we  obtain the center velocity $u$ using the constraint equation $u=rU\nu(a)$.
\end{proposition}
\begin{proof}
The given equation follows from (\ref{w}) by simple manipulations.
\end{proof}

\subsection{An alternative form of the rolling equation}
We give here an alternative form of the rolling equation that will be better adapted for our later purposes. Instead of using the full angular velocity matrix $U$, we write Newton's equation as a system involving the center velocity $u$  and a tensor on $\mathcal{N}$ which we call the {\em tangential spin matrix} (not to be confused with spin structures in the standard sense of Riemannian geometry!)

\begin{lemma}\label{lemma decomposition}
Let $\nu$ be a unit vector in $\mathbb{R}^m$  and $\Pi$ the orthogonal projection to the codimension-$1$ subspace  perpendicular to $\nu$.
Then any $V\in \mathfrak{so}(m)$ can be written as
$$V= \Pi V \Pi + \nu \wedge V\nu. $$
This is an orthogonal decomposition with respect to the trace inner product.
\end{lemma}
\begin{proof}
Let $\Pi^\perp$ be the orthogonal projection to the line spanned by $\nu$.
Any linear transformation $V$ of $\mathbb{R}^m$ can be written as
$$ V= \Pi V\Pi + \Pi^\perp V\Pi + \Pi V\Pi^\perp + \Pi^\perp V\Pi^\perp. $$
The last term is $0$ if $V$ is skew symmetric and it is easily seen that
$$\Pi V\Pi^\perp w = (w\cdot \nu) \nu, \ \ \Pi^\perp V\Pi w = - (V\nu)\cdot w \nu. $$
The claim now follows from the definition of the cross-product $\wedge$. Orthogonality is an easy verification. 
\end{proof}

It follows from the lemma that if $(u,U)\in \mathfrak{se}(m)$ defines a state at configuration $g=(a,A)$ that satisfies the rolling constraint $u=rU\nu(a)$,
then 
$$ U=\Pi_a U \Pi_a + \frac1r \nu(a)\wedge u.$$
Rather than using $U$ to describe the state (from which we obtain $u$ using the constraint equation), we  
use its tangential part $$S_a:=\Pi_a U\Pi_a$$
and $u$, from which the other components of $U$ can be derived. We will refer to $S_a$ as the body's {\em tangential angular velocity} or {\em tangential spin}. (We give the same name to both $S$ and the related quantity $\mathcal{S}$.)
This leads to the following simple observation.
\begin{proposition}\label{U and S}
Under the no-slip constraint, the angular velocity matrix $U$ satisfies 
$$U=S_a + \nu(a)\wedge \frac{u}{r}$$
where $S_a$ is the tangential spin  at configuration $g=(a,A)$ and $u$ is the   velocity of the center point $a\in \mathcal{N}$. 
\end{proposition}
 We wish next to rewrite the equation of motion in terms of $S_a$ and $u$ rather than $U$. This requires first relating the time derivative of $S_a$ 
 as a function taking values in $\mathfrak{so}(m)$, and its covariant derivative as a tensor field on $\mathcal{N}$ along paths. 
 \begin{lemma}\label{lemma Sdot}
 Let $a(t)$ be a differentiable curve in the hypersurface $\mathcal{N}\subset \mathbb{R}^m$, write $u=\dot{a}$, and let $S(t):T_{a(t)}\mathcal{N}\rightarrow T_{a(t)}\mathcal{N}$ be a differentiable field of symmetric linear maps along $a(t)$. Let as before $\nu$ denote a unit normal vector field  on $\mathcal{N}$ and $\mathbb{S}$ the shape operator, $\mathbb{S}_a v =- D_v \nu$,  where $D_u$ denotes ordinary directional derivative of vectors in $\mathbb{R}^m$ and  $v\in T_a\mathcal{N}$. Let $\nabla$ denote the Levi-Civita connection on the hypersurface. Then
 $$\dot{S} = \frac{\nabla S}{dt} + \nu(a) \wedge S\mathbb{S}_a u. $$
 \end{lemma}
 \begin{proof}
 Let us introduce a local orthonormal frame of differentiable vector fields on $\mathcal{N}$: $E_1, \dots, E_{m-1}$. Then
 $\{E_i\wedge E_j: 1\leq i< j\leq m-1\}$ is a basis of $\mathfrak{so}_a(\mathcal{N})$ and we may write
 $$S= \sum_{i<j} s_{ji}E_i\wedge E_j. $$
Using inner product notation $\langle\cdot, \cdot\rangle$ for the standard dot product,
 $$ \dot{E}_j = D_uE_j = \sum_{i=1}^{m-1}\langle E_i, D_uE_j\rangle E_i + \langle \nu(a), D_uE_j\rangle \nu(a) =\nabla_uE_j -  \langle D_u\nu(a), E_j\rangle \nu(a).$$
 Since $\mathbb{S}_a$ is symmetric, $-  \langle D_u\nu(a), E_j\rangle= \langle \mathbb{S}_a u, E_j\rangle=\langle u, \mathbb{S}_a E_j\rangle $ and  we obtain 
 $$D_u{E}_j = \nabla_u E_j +\langle u, \mathbb{S}_a E_j\rangle  \nu(a). $$
 It follows that
 \begin{align*}D_u(E_i\wedge E_j)&= \left(\nabla_uE_i + \langle u,\mathbb{S}_a E_i\rangle \nu(a)\right)\wedge E_j  + E_i\wedge\left(\nabla_uE_j  + \langle u,\mathbb{S}_a E_j\rangle \nu(a)\right)\\
 &=\nabla_u \left(E_i\wedge E_j\right) + \nu(a)\wedge \left( \langle \mathbb{S}_au, E_i\rangle E_j  -  \langle \mathbb{S}_au, E_j\rangle E_i  \right)\\
 &=\nabla_u \left(E_i\wedge E_j\right)  + \nu(a) \wedge \left[\left(E_i\wedge E_j\right) \mathbb{S}_a u\right].
 \end{align*}
 Finally,
 \begin{align*}\dot{S} &= \sum_{i<j} \left[\dot{s}_{ji} E_i\wedge E_j + s_{ji} D_u(E_i\wedge E_j)\right]\\
 &= \sum_{i<j} \left[\dot{s}_{ji} E_i\wedge E_j + s_{ji} \nabla_u(E_i\wedge E_j)\right] +\nu(a) \wedge \sum_{i<j} s_{ji} E_i\wedge E_j \mathbb{S}_a u\\
 &=\frac{\nabla S}{dt} + \nu(a)\wedge S\mathbb{S}_a u
 \end{align*}
 as claimed.
 \end{proof}

  Notational simplification is achieved by introducing the following linear map, which we still call (tangential) {\em spin}:
 $$ \mathcal{S} = r\eta S$$
 where $\eta=\gamma/\sqrt{1+\gamma^2}.$ One  advantage of $\mathcal{S}$ over $S$ is  that the kinetic energy of a  state (satisfying the
 no-slip constraint) represented by $(\mathcal{S}, u)$ becomes, up to a multiplicative constant, $|u|^2 + \frac12 \text{Tr}(\mathcal{S}\mathcal{S}^\intercal)$.
 The final equation of motion will also take a more symmetric form.

  We are now ready to express the equation of motion in terms of the spin matrix $\mathcal{S}$ and the center velocity $u$,
  eliminating any reference to the non-holonomic constraint condition.   The proof 
  of this proposition also yields the proof of Theorem \ref{Newton}.
 \begin{proposition}\label{roll equations proposition}
The rolling motion with  hypersurface (of ball  centers) $\mathcal{N}$ under the no-slip constraint   satisfies the system of equations 
\begin{align*}
\frac{\nabla u}{dt} &= -\eta \mathcal{S} \mathbb{S}_a u \\
\frac{\nabla \mathcal{S}}{dt} &= \eta \left(\mathbb{S}_a u\right) \wedge u  
\end{align*}
where $u=\dot{a}\in T_a\mathcal{N}$ is the velocity of the center of the ball and $\mathcal{S}$ is the tangential spin. 
Here $\nabla$ is the ordinary Levi-Civita connection of the hypersurface with the   Riemannian metric induced by restriction of the dot product in $\mathbb{R}^m$.
When the moment of inertia is zero ($\eta=0$) the system reduces to geodesic motion on $\mathcal{N}$  with parallel tangential spin. 
 \end{proposition}
 \begin{proof}
The proof simply amounts to rewriting the terms of the main equation of Proposition 
\ref{propo_roll_equation} first using  $S$ and then $\mathcal{S}$. We limit ourselves here to listing the expressions already obtained that go into the rewriting.  First note that Proposition \ref{U and S}   gives
\begin{equation}\label{equation 1} \left(U\mathbb{S}_a U \nu(a)\right)\wedge \nu(a) =-\frac1r  \nu(a) \wedge S\mathbb{S}_a u.\end{equation}
From this and Proposition \ref{propo_roll_equation}  we find
\begin{equation}\label{equation 2}
\dot{U}= F(a) +\frac1{1+\gamma^2} \nu(a) \wedge S\mathbb{S}_a u.
\end{equation}
Differentiating the constraint equation $u= rU\nu(a)$ in $t$ gives $$\dot{u}= r\dot{U} \nu(a)- rU\mathbb{S} u.$$ This and the main equation of Proposition
\ref{propo_roll_equation} yield the relation
\begin{equation}\label{equation 3} \frac{\nabla u}{dt} = r F(a)\nu(a) - r\eta^2 S\mathbb{S}_a u.\end{equation}
Proposition \ref{U and S} and Lemma \ref{lemma Sdot} yield
\begin{equation}\label{equation 4}
\dot{U} = \frac{\nabla S}{dt} +\frac1{1+\gamma^2} \nu(a) \wedge S\mathbb{S}_au+ \frac{u}r \wedge \mathbb{S}_a u  + \nu(a)\wedge F(a)\nu(a).
\end{equation}
Combining this   with Equation (\ref{equation 2}) results in
\begin{equation}\label{equation 5}
\frac{\nabla S}{dt} = -\frac{u}{r}\wedge \mathbb{S}_au + F(a)-\nu\wedge F(a)\nu(a) =  -\frac{u}{r}\wedge \mathbb{S}_au + \Pi_a F(a) \Pi_a.
\end{equation}
Closer inspection shows that 
 Equations (\ref{equation 2}) and (\ref{equation 5}) are, in fact,  equivalent to the main equation of Proposition \ref{propo_roll_equation} together with
the constraint equation.
The definition of   $\mathcal{S}$ finally  gives the wished for equations. 
 \end{proof}
 
 It is worthwhile writing down the rolling equations in the special case of dimension $m=3$. In this case,  $\mathcal{N}$ is a surface and
  $\mathfrak{so}(\mathcal{N})$ is a  trivial line bundle. Let $J_a$ denote rotation on $T_a\mathcal{N}$ counterclockwise (given the  orientation defined by $\nu(a)$) by $\pi/2$. In terms of a local, positive, orthonormal frame of vector fields $E_1, E_2$, we have $J=E_1\wedge E_2$.
  Then the tangential spin matrix has the form $\mathcal{S}=\mathcal{s} J$. Note that $J$ is a parallel tensor,  $\nabla J=0$. It makes sense to express the rolling equations on the product manifold $M=\mathcal{N}\times \mathbb{R}$ where   $\vartheta$ given by the second factor may be interpreted  
  as an overall amount of rotation that is related to $\mathcal{s}$ through the equation $\mathcal{s}=\dot{\vartheta}$. Notice that we are redefining the symbol $M$, which previously was used to represent the configuration manifold of the rolling ball. In this new $M$,   the no-slip constraint is already built in. The rolling equations (in the absence of forces) take now the form
  \begin{equation}\label{dim 3}\frac{\nabla u}{dt}= -\eta \mathcal{s} J \mathbb{S}_a u, \ \ \dot{\mathcal{s}}=\eta \alpha(\mathbb{S}_a u, u) \end{equation}
 where $(a,\vartheta)\in M$, $(u,\mathcal{s})=(\dot{a},\dot{\vartheta})$, and $\alpha$ is the area $2$-form on $\mathcal{N}$. Note that
 $\alpha =\epsilon_1\wedge \epsilon_2$ where $\epsilon_i$, $i=1,2$,  is the dual frame and $\wedge$ is here the standard wedge product of differential forms.

 \section{Pancake surfaces and the billiard limit} \label{bill limit section}

 Let $P$  be an open, connected subset of  $\mathbb{R}^{k}$ with piecewise  smooth (manifold) boundary $P_0=\partial P$; it is imbedded in $\mathbb{R}^{k+1}$ so that
 the last coordinate of each of its points equals $0$.    Then $P_0$ is a codimension $2$ imbedded submanifold of $\mathbb{R}^{k+1}$. We will refer to $P$ 
 as the flat {\em plate} and $P_0$ as the {\em edge} of the plate.
  The motion of the rolling ball around the edge $P_0$ is not in general easy to describe because the curvature of $P_0$ influences the rolling trajectory in subtle ways.  For example, when $P$ is a disc in $\mathbb{R}^3$,    the  ball may roll  all the way to the other side of $P$
 upon reaching 
 the circular edge 
  or it may turn around midway and return to the same side from which it started, depending on the choice of initial conditions.
  Our goal here is to show that, in the limit as the radius approaches $0$, 
orbit behavior   is fully described by the straight edge Example \ref{around edge}.  This  is  natural since a very small ball  should not     feel  the curvature of $P_0$.  (We assume  the ball's trajectory does not go through points where $P_0$ is not  smooth.)
The end result is a billiard system in which the reflection map $(u^-, W^-)\mapsto (u+, W^+)$ is as described in Example \ref{around edge}.
The proof is not entirely straightforward, however, because the equations of motion become singular as $r\rightarrow 0$, and the limit needs to be taken with some care. The main result is contained in the   theorem stated after the following definition. For simplicity, we assume from here on that
$P_0$ is smooth.
  
      \begin{figure}[htbp]
\begin{center}
\includegraphics[width=4in]{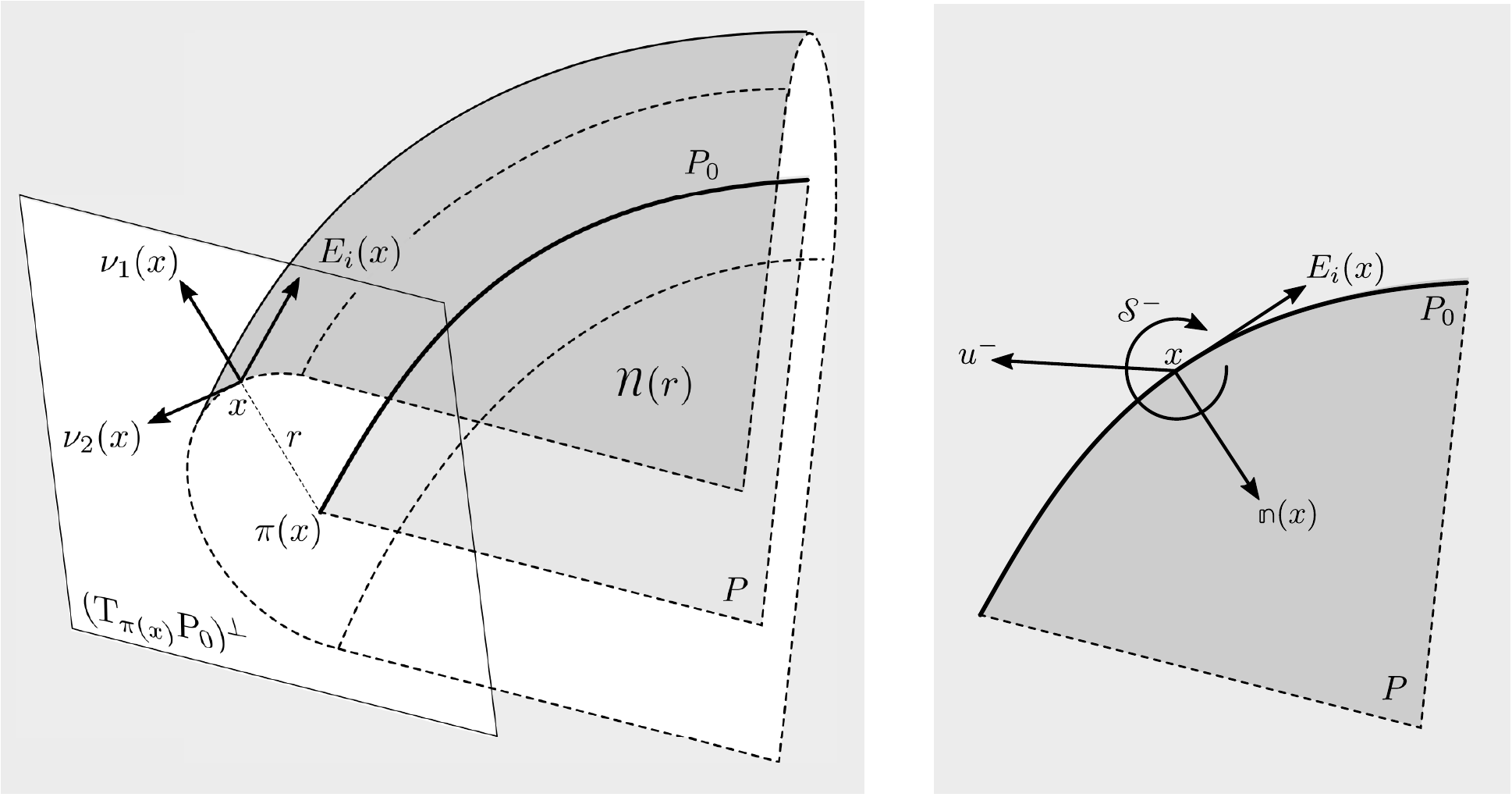}\ \ 
\caption{{\small On the left-hand side: hypersurface $\mathcal{N}=\mathcal{N}(r)$ associated to the manifold $P$. On the right-hand side: in the limit $r\rightarrow 0$,
one obtains a (non-holonomic) billiard system with collision map $\mathcal{C}: (u^-,W^-)\mapsto (u^+,W^+)$, where
the $W^\pm$ are related to the $\mathcal{S}^\pm$ as explained in Example \ref{around edge}.  }}
\label{tube}
\end{center}
\end{figure}

  \begin{definition}[The non-holonomic collision map]\label{coll map}
  Let $P$ be an $k$-dimensional flat plate in $\mathbb{R}^{k+1}$ with    smooth manifold boundary $P_0$. At any 
  $x\in P_0$ we define the vector space $V_x=T_xP_0\oplus T_xP_0\oplus \mathbb{R} \mathbbm{n}(x)$ consisting of vectors
  $(W, u_0, u^\perp)$, where $u=u_0+u^\perp\in T_xP$ and $\mathbbm{n}(x)$ is the inward pointing unit normal vector to $P_0$. (See Figure \ref{tube}.) Then 
    the {\em non-holonomic collision map}  $\mathcal{C}_x:V_x\rightarrow V_x$ is defined as
    $$ \mathcal{C}_x(W,u_0,u^\perp)= \left(\sin(\pi\eta)  u_0 -\cos(\pi\eta) W,  \cos(\pi\eta)  u_0 + \sin(\pi\eta) W, 
     - u^\perp\right).$$
     Note that, when $\eta=0$, $\mathcal{C}_x$ reflects $u$ specularly and reverses the sign of $W$. 
  \end{definition}

In the above definition, we interpret $u\in T_xP$ as the velocity of a point-mass and $W$ as the essential components of a tangential angular 
velocity tensor $\mathcal{S}$ in the following sense: $W=\mathcal{S}\mathbbm{n}(x)$. This gives rise to a billiard system in which the point particle 
possesses, in addition to its velocity, a kind of spin velocity  that affects the result of the collision with the boundary of $P$. 
It is natural to think that the particle is an infinitesimal rolling ball  with rotationally symmetric mass distribution whose moment of inertia is specified by  $\eta$. We recall that  the moment of inertia parameters $\eta$ and $\gamma$ do not depend on the radius of the ball.

  \begin{theorem}\label{limit r proposition}
  Let $P$ be an  $k$-dimensional  flat plate   in $\mathbb{R}^{k+1}$ with smooth boundary $P_0$ whose principal curvatures (as a hypersurface in $P$) are uniformly bounded. In the limit when the radius of the ball approaches $0$, solutions of the rolling ball equation   have the following description:
  On $P\setminus P_0$ the point-mass moves with constant velocities  $u$ and $\mathcal{S}$; upon reaching a boundary point $x\in P_0$, the vector $(u,W)\in T_xP\oplus  T_xP_0$, where $W=\mathcal{S}\mathbbm{n}(x)$, 
 undergoes a reflection   according to the non-holonomic collision map $\mathcal{C}_x$ of Definition \ref{coll map}. 
  \end{theorem}
 \begin{proof}
To begin, let us assume that the radius $r$ of the ball is sufficiently small so that the map $\pi:\mathcal{N}(r)=\mathcal{N} \rightarrow P$ that associates to each $x\in \mathcal{N}(r)$ the closest point in $P$ is well defined.  This is possible due to the assumption that the principal curvatures of $P_0$
are bounded. The hypersurface $\mathcal{N}$
is piecewise smooth and consists of the union of two parallel copies of $P$,  lying $2r$ apart from each other,  and half the boundary of the tube of radius $r$ centered around $P_0$.
We call the two copies of $P$ the two {\em sheets} of $\mathcal{N}$ and the half-tube the {\em curved part} of $\mathcal{N}$.

Let  $x, u, \mathcal{S}$ be initial conditions for the rolling equation, where $x\in \mathcal{N}$ is a point on the interface
where the curved part of $\mathcal{N}$  
meets the flat sheets. Note that this interface is the union of two  diffeomorphic copies of $P_0$.  
Here   $u\cdot \mathbbm{n}(x)<0$,
so the center  velocity $u$ points towards the curved part of $\mathcal{N}$); $\mathcal{S}$ is the tangential angular velocity tensor.  Set $W:=\mathcal{S}\mathbbm{n}(x)$.  
Let $u(t)=\dot{x}(t)$ and $\mathcal{S}(t)$ satisfy the rolling equations
$$ \frac{\nabla u}{dt} = -\eta \mathcal{S} \mathbb{S}_x u, \ \ \frac{\nabla \mathcal{S}}{dt} =\eta (\mathbb{S}_xu)\wedge u$$
with the given initial conditions. We follow the solution from time $0$ till the moment (if it happens) when the ball reaches the interface submanifold again.

The shape operator $\mathbb{S}$ naturally becomes singular as $r$ approaches $0$. In fact, on the intersection of the curved part of $\mathcal{N}$ with
the $2$-plane perpendicular to $T_{\pi(x)}P_0$ (see the left-hand side of Figure \ref{tube}), the principal curvature is $-1/r$. We call this intersection the
{\em meridian of $\mathcal{N}$} at $x$. 
This produces a discontinuity
of velocities at the limit. It is also to be expected that the duration of the rolling on the curved part of $\mathcal{N}$ approaches $0$ in the limit. With these issues in mind, we transform the original equations of motion by making a time change and   applying an appropriate homothety. The resulting 
system will be of the kind considered in Example \ref{around edge} (giving the rolling of a finite radius ball on a straight edge). 

Here are some of the details.
Let $c^2$ be the square norm of $(u,\mathcal{S})$, 
a quantity proportional to the energy of the initial condition.
 Introduce a new time given by $\tau=\frac{c}{r}t$ and define the homothety $h:x\in\mathbb{R}^m\rightarrow x/r\in \mathbb{R}^m$.
 Let $\overline{\mathcal{N}}$ be the image of $\mathcal{N}(r)$ under $h$, appropriately translated so the projection $\pi(x)$ of the initial point
 lies at the origin.  Note that, as $r$ approaches $0$, $\overline{\mathcal{N}}$ looks increasingly like the straight edge situation of Example \ref{around edge}.
Now define
$$\bar{x}(\tau) = h(x(ct/r)), \ \ \bar{\mathcal{S}}(\tau) = h\left(\frac{r}{c} \mathcal{S} (ct/r)\right). $$
For any given value of $r$, the rolling equations turn into
$$ \frac{\nabla \bar{u}}{d\tau} = -\eta \bar{\mathcal{S}} \bar{\mathbb{S}}_{\bar{x}} \bar{u}, \ \  \frac{\nabla\bar{\mathcal{S}}}{d\tau}= \eta\left(\bar{\mathbb{S}}_{\bar{x}}\bar{u}\right)\wedge \bar{u}.$$
where the new shape operator $\mathbb{S}$ at $\bar{x}(\tau)$ equals $r\mathbb{S}$ at $x(ct/r)$. The norm of the new initial velocities $(\bar{u}, \bar{\mathcal{S}})$ is $1$ for all $r$.
The  principal curvature  on the meridian circles becomes $-1$ for all $r$, and the other principal curvatures  approach $0$. In the limit,
this shape operator becomes $-E^\flat\otimes E$ where $E$ is a unit vector field tangent to the meridian circle and $E^\flat$ is its dual vector relative
to the dot-product.  
 
The meridian circles are geodesics so $\nabla_EE=0$,  and $E$ has constant norm, so
 $E\cdot \nabla_{v} E=0$ for any tangent vector $v$. Writing  $\bar{u}^\perp$ for the  component of $\bar{u}$ perpendicular to $E$, we obtain
 $\bar{u}\cdot \nabla_{\bar{u}}E =  \bar{u}^\perp \cdot \nabla_{\bar{u}^\perp}E.$
Then, using the equations of motion,
$$ \frac{d}{d\tau} \bar{u}\cdot E= \frac{\nabla \bar{u}}{d\tau}\cdot E + \bar{u}\cdot \nabla_{\bar{u}}E= -\eta E\cdot \left(\bar{\mathcal{S}} \bar{\mathbb{S}}\bar{u}\right) + \bar{u}^\perp\cdot \nabla_{\bar{u}^\perp}E.$$
As $r$ approaches $0$, $\bar{\mathbb{S}}\bar{u}$ converges to a vector parallel to $E$; since $\bar{\mathcal{S}}$ is skew-symmetric, the term
$E\cdot \left(\bar{\mathcal{S}} \bar{\mathbb{S}}\bar{u}\right)$ approaches $0$. Notice that $E$ is normal to the
isometric copies of the rescaled $P_0$, so the quantity $\nabla_{\bar{u}^\perp}E$ is the negative of the shape operator of this submanifold. 
Thus the term $\bar{u}^\perp\cdot \nabla_{\bar{u}^\perp}E$ also approaches $0$ due to the assumption that the principal curvatures of $P_0$ are
 bounded. 
The conclusion is that, in the limit, $\bar{u}\cdot E=\mu$ is a constant of motion and, 
in any fixed neighborhood of the initial (rescaled) point, the hypesurface $\overline{\mathcal{N}}$ approaches that  of
Example \ref{around edge}.   By introducing an orthonormal frame $E_1, \dots, E_{m-2}$ of parallel vector fields tangent to the rescaled (and straightened) $P_0$,  and using $\bar{\mathbb{S}}E_i=0$, we obtain the system of equations
$$\frac{d}{dt}(E_i\cdot \bar{u})=\eta \mu (E_i\cdot\bar{\mathcal{S}} E), \ \ \frac{d}{dt}(E_i\cdot \bar{\mathcal{S}}E)=-\eta \mu E_i\cdot \bar{u}.$$ 
But these are precisely the equations of Example \ref{around edge}. By reversing the rescaling on velocities we obtain from the conclusion of that example the
collision map $\mathcal{C}$ we are after.
 \end{proof}

 \end{document}